\documentclass[11pt,a4paper]{amsart}
\usepackage[english]{babel} 
\usepackage{hyperref} 
\usepackage{amsfonts,amscd,amsmath,amssymb,amsthm} 
\usepackage{xparse} 
\usepackage{dsfont} 
\usepackage[utf8]{inputenc} 
\usepackage{xcolor}
\usepackage{tikz}

\numberwithin{equation}{section}

\newtheorem{theorem}{Theorem}[section]{\bf}{\it}
\newtheorem{proposition}[theorem]{Proposition}{\bf}{\it}
\newtheorem{corollary}[theorem]{Corollary}{\bf}{\it}
\newtheorem{lemma}[theorem]{Lemma}{\bf}{\it}

\theoremstyle{definition}
\newtheorem{definition}[theorem]{Definition}
\newtheorem{example}[theorem]{Example}

\theoremstyle{remark}
\newtheorem{remark}[theorem]{Remark}

\theoremstyle{plain}
\newcounter{countertheoremintro}
\newtheorem{introtheorem}[countertheoremintro]{Theorem}

\newcommand{\idempotents}{E}

\newcommand{\unifroealg}{C_u^*(S, d)}

\newcommand{\lrel}{\mathrel{\mathcal{L}}}
\newcommand{\rrel}{\mathrel{\mathcal{R}}}
\newcommand{\drel}{\mathrel{\mathcal{D}}}

\newcommand{\prop}{\text{prop}}

\newcommand{\diam}{\operatorname{diam}}
\newcommand{\sch}[1]{S\Gamma({#1})} 
\newcommand{\asdim}{\operatorname{asdim}}
\newcommand{\sym}[1]{\mathcal I_{#1}} 
\newcommand{\id}{\operatorname{id}}
\newcommand{\ol}{\overline}

\title[Quasi-countable inverse semigroups as metric spaces]{Quasi-countable inverse semigroups as metric spaces, and the uniform Roe algebras of locally finite inverse semigroups}

\author[Yeong Chyuan Chung]{Yeong Chyuan Chung}
\address{School of Mathematics, Jilin University, Changchun 130012, Jilin, P.R. China}
\email{chungyc@jlu.edu.cn}

\author[Diego Mart\'{i}nez]{Diego Mart\'{i}nez $^{1}$}
\address{Mathematisches Institut, WWU M\"{u}nster, Einsteinstr. 62, 48149 M\"{u}nster, Germany}
\email{diego.martinez@uni-muenster.de}

\author[N\'{o}ra Szak\'{a}cs]{N\'{o}ra Szak\'{a}cs}
\address{School of Mathematics, University of Manchester, Manchester M13 9PL, UK}
\email{nora.szakacs@manchester.ac.uk}

\date{\today}
\thanks{{$^{1}$} Funded by the Deutsche Forschungsgemeinschaft (DFG, German Research Foundation) under Germany’s Excellence Strategy – EXC 2044 – 390685587, Mathematics Münster – Dynamics – Geometry – Structure; the Deutsche Forschungsgemeinschaft (DFG, German Research Foundation) – Project-ID 427320536 – SFB 1442, and ERC Advanced Grant 834267 - AMAREC}

\begin{document}

  \begin{abstract}
    Given any quasi-countable, in particular any countable inverse semigroup $S$, we introduce a way to equip $S$ with a proper and right subinvariant extended metric. This generalizes the notion of proper, right invariant metrics for discrete countable groups. Such a metric is shown to be unique up to bijective coarse equivalence of the semigroup, and hence depends essentially only on $S$. This allows us to unambiguously define the uniform Roe algebra of $S$, which we prove can be realized as a canonical crossed product of $\ell^\infty(S)$ and $S$. We relate these metrics to the analogous metrics on Hausdorff \'{e}tale groupoids.

    Using this setting, we study those inverse semigroups with asymptotic dimension $0$. Generalizing results known for groups, we show that these are precisely the locally finite inverse semigroups, and are further characterized by having strongly quasi-diagonal uniform Roe algebras. We show that, unlike in the group case, having a finite uniform Roe algebra is strictly weaker and is characterized by $S$ being locally $\mathcal L$-finite, and equivalently sparse as a metric space.
  \end{abstract}
  \maketitle

  \section{Introduction}\label{sec:intro}
  
  Large scale geometry (or coarse geometry) is, roughly speaking, the study of metric spaces viewed from afar, that is, two spaces that look the same from a great distance are considered equivalent. The idea was already present in classical results such as Mostow's rigidity theorem~\cite{Mos68} and Wolf's work on growth of groups~\cite{Wolf68}, and has more recently played a key role in the study of finitely generated groups as metric spaces. 
  
  Finitely generated discrete groups are naturally equipped with a metric called the \emph{word metric}, which is defined as the path metric in (any of) their Cayley graphs. The large scale geometry of this metric is group invariant, and is closely tied to the algebraic properties of the group. In the early 90s, Gromov initiated a program for the systematic study of large scale properties of finitely generated groups \cite{Gro81,Gro93} which led to the development of \emph{geometric group theory}, and has applications to topology and index theory~\cite{Roe03,Yu98}, as well as operator algebras and noncommutative geometry~\cite{Yu95,Yu00}. 
  
  The idea that groups can be viewed as metric spaces generalizes to countably generated groups. The necessary properties of the metric that ensure it is connected to the algebraic structure of the group are \emph{right invariance}, that is, $d(g, h) = d(gx, hx)$ for every $g, h, x \in G$, and \emph{properness}, meaning every ball of finite radius is finite. It is a classical result of Birkhoff and Kakutani from the 30s that any first countable Hausdorff topological group is metrizable, and the metric can be chosen to be right invariant, and Struble \cite{Stru74} observed that if the group is locally compact and second countable, the metric is also proper. In particular, countable discrete groups can be equipped with a proper, right invariant metric. Dranishnikov and Smith observe \cite{DS06} that this metric is unique up to \emph{bijective coarse equivalence}, which allows them to unambiguously define the asymptotic dimension of non-finitely generated subgroups of finitely generated groups. Other well-studied large scale properties, such as amenability, property A, or coarse embeddability into a Hilbert space, also naturally generalize to this setting.
  

  Motivated by index theory on manifolds, Roe introduced two C*-algebras associated to metric spaces \cite{Roe03}, now called the \emph{Roe algebra} and the \emph{uniform Roe algebra}, which capture the large scale geometry of the space. If the metric space comes from a countable group, the uniform Roe algebra can also be constructed as a crossed product from the right action of the group on itself. This provides a three-way interplay between group theory, large scale geometry and operator algebras, which lies at the heart of the research on the coarse Baum-Connes conjecture and the Novikov conjecture \cite{Yu95, Yu98, Yu00}.

%
  
  In the past few years, some of the ideas above have been extended from groups to \emph{inverse semigroups} \cite{GSSz, LM21}, one of the most important generalizations of groups. These are the semigroups where each element $s$ has a unique inverse $s^\ast$ such that $ss^\ast s=s$ and $s^\ast ss^\ast=s^\ast$. They originally emerged in the 50s as the algebraic abstraction of partial symmetries, and an extensive theory of abstract inverse semigroups was developed in the following decades (see~\cite{L98} and references therein). At the same time, inverse semigroups kept appearing in various guises all over mathematics: in the context of C*-algebras, tilings, model theory, linear logic and combinatorial group theory \cite{L98, law2}, among others. Constructing C*-algebras from inverse semigroups has become a central theme in operator algebras \cite{exel}.
  
  Equipping inverse semigroups with a word metric comes with several technical caveats. The main difficulties are that multiplication in inverse semigroups is typically not cancellative, and the Cayley graph of an inverse semigroup may not be strongly connected; in particular, edges do not necessarily come in inverse pairs. As an extreme case, the inverse semigroup may have a $0$, which is a sink reachable from every vertex.  
  However, for finitely generated inverse semigroups, considering the path metric on the strong components of the Cayley graph and defining different components to be at infinite distance apart yields an \emph{extended} metric that is naturally connected to properties of the semigroup \cite{GSSz, LM21}. Moreover, the uniform Roe algebra $\unifroealg$ associated to this extended metric space can also be obtained as a crossed product from the right action of the semigroup on itself \cite{LM21}, and large scale invariant amenability-type conditions (such as domain measurablility and property A, see~\cite{ALM19,LM21}) correspond to C*-properties of the uniform Roe algebra.

  In this paper, our goal is to extend this approach to countable inverse semigroups as was done in the group case by equipping them with a metric in a coarse unique way, and study inverse semigroups with asymptotic dimension $0$ from both an algebraic and C* point of view.
  
  To define the metric, the first step is to find the analogous properties to right invariance and properness in the inverse semigroup settings. Right invariance itself cannot be assumed -- going back to the extreme case when $S$ has a zero element $0 \in S$, it would imply $d(s,t)=d(s0,t0)=0$ for all pairs $s,t \in S$. Hence, right invariance is replaced by \emph{right subinvariance}, meaning we assume $d(sx, tx) \leq d(s, t)$ for all $s, t, x \in S$.       

 It is easy to see that this condition reduces to right invariance when $S$ happens to be a group. Defining properness is a bit more subtle, and we will leave its discussion to Definition~\ref{def:metric}, where all these notions are introduced. For now, let us record the following theorem, which is one of the main contributions of the paper.
  \begin{introtheorem}[see Theorem~\ref{thm:metric}] \label{introthm:metric}
    Let $S$ be an inverse semigroup. Then the following statements are equivalent:
    \begin{enumerate}
      \item $S$ is quasi-countable, i.e., there is some countable $F \subset S$ such that $S = \langle F \cup E \rangle$, where $E$ is the set of idempotents of $S$;
      \item $S$ admits a proper and right subinvariant uniformly discrete extended metric whose components are exactly the $\mathcal{L}$-classes. 
    \end{enumerate}
    Moreover, such a metric is unique up to bijective coarse equivalence.
  \end{introtheorem}
  In particular, the above theorem provides an essentially unique metric to any countable inverse semigroup. It is worth mentioning that these metrics correspond to coarse metrics on the universal groupoid of $S$ in the sense of Ma and Wu~\cite{MW21}, who define length functions on locally compact \'{e}tale Hausdorff groupoids. In Subsection~\ref{groupoid-metric} we give a detailed discussion of these topics, see Theorem~\ref{thm:metric-groupoids-relations} for a precise relationship between the metrics studied in this text and those of~\cite{MW21}.

  Before recording more contributions of the paper, we would like to highlight the freedom that working with inverse semigroups entails. The following theorem shows that any metric space can appear as some $\mathcal{L}$-class of some countable inverse semigroup.
  \begin{introtheorem}[see~\ref{thm:wobbling-inv-sem}] \label{introthm:examples}
    Let $(X, d_X)$ be an infinite non-extended metric space of bounded geometry. Then there is a countable inverse semigroup $S$ and an $\lrel$-class $L \subset S$ such that $(L, d_S)$ is bijectively coarsely equivalent to $(X, d_X)$, where $d_S$ is any proper and right subinvariant metric on $S$.
  \end{introtheorem}

  Extending the analogous result of~\cite{LM21}, we prove in Theorem~\ref{thm:roealg:unique} that the uniform Roe algebra associated to a proper and right subinvariant metric can also be recovered from the natural left action of the semigroup on itself.
  \begin{introtheorem}[see Theorem~\ref{thm:roealg:unique}] \label{introthm:roealg}
    Let $S$ be a quasi-countable inverse semigroup, and let $d$ be a proper and right subinvariant metric on $S$. Then $\ell^\infty(S) \rtimes_r S$ and $\unifroealg$ are $*$-isomorphic as C*-algebras, where the action of $S$ on $\ell^\infty(S)$ is the canonical action by left-translation.
  \end{introtheorem}

	Using this framework, we then characterize inverse semigroups with \emph{asymptotic dimension} $0$. Asymptotic dimension was introduced by Gromov as a large scale invariant notion of dimension, and is analogous to Lebesgue's covering dimension of a topological space. Of particular interest to us are those inverse semigroups $S$ such that $(S, d)$ has asymptotic dimension $0$. For countable groups, \cite{Smith} shows asymptotic dimension $0$ to be equivalent to being \emph{locally finite}, and \cite{Scar17} shows this is equivalent to its uniform Roe algebra being \emph{finite}, which \cite{LL18} shows to be equivalent to several other properties of the C*-algebra, including being \emph{quasi-diagonal} or \emph{stably finite}. 
	
	We investigate the relationship between these notions in the more general setting of inverse semigroups. In Theorem~\ref{thm:asymdim-zero} we show that like in the group case, asymptotic dimension $0$ is equivalent to local finiteness. However, when it comes to the  C*-algebraic results, the characterizations change. Asymptotic dimension $0$ corresponds to the uniform Roe algebra being local AF, or equivalently, strongly quasi-diagonal. The C*-algebraic conditions seen for groups are strictly weaker in this setting, and in Theorem~\ref{thm:boxspace} are shown to be equivalent to the metric being \emph{sparse}, and the inverse semigroup being \emph{locally $\mathcal L$-finite}.
	
  \begin{introtheorem}[see Theorem~\ref{thm:asymdim-zero}] \label{introthm:asymdim-zero}
    Let $S$ be a quasi-countable inverse semigroup equipped with its unique proper and right subinvariant metric $d$. The following statements are equivalent:
    \begin{enumerate}
      \item $S$ is locally finite.
      \item $(S, d)$ has asymptotic dimension $0$.
      \item $\unifroealg$ is local AF.
      \item $\unifroealg$ is strongly quasi-diagonal.
    \end{enumerate}
  \end{introtheorem}

 \begin{introtheorem} [see Theorem~\ref{thm:boxspace}]
	Let $S$ be a quasi-countable inverse semigroup equipped with its unique proper and right subinvariant metric $d$. The following statements are equivalent:
	\begin{enumerate}
		\item $S$ is locally $\lrel$-finite.
		\item $(S, d)$ is sparse.
		\item $\unifroealg$ is quasi-diagonal.
		\item $\unifroealg$ is stably finite.
		\item $\unifroealg$ is finite.
	\end{enumerate}
\end{introtheorem}

The interest of the latter theorems lies in the observation that quasi-diagonality, as a property of C*-algebras is, generally speaking, poorly understood. It is a very powerful tool to study the structure of a C*-algebra that imposes strong conditions, but we do not have many ways to check whether a given C*-algebra is quasi-diagonal or not. In particular, observe that the groundbreaking results of~\cite{TWW17} do not apply here, since the uniform Roe algebras we study are almost never separable.

  We end the introduction with an account of how the text is organized. Section~\ref{sec:pre} collects all the background notions needed throughout the text about inverse semigroups, coarse geometry or C*-algebras. Section~\ref{sec:metrics} then introduces when a metric is proper and right subinvariant, and proves Theorem~\ref{introthm:metric} (see Theorem~\ref{thm:metric}), and in Subsection~\ref{subsec:ex-metrics} we prove Theorem~\ref{introthm:examples} (see Theorem~\ref{thm:wobbling-inv-sem}). Section~\ref{sec:roealg} then introduces the uniform Roe algebra of $S$, and proves Theorem~\ref{introthm:roealg} (see Theorem~\ref{thm:roealg:unique}). Lastly, in Section~\ref{sec:asymdim-lf} we characterize those inverse semigroups of asymptotic dimension $0$, proving Theorem~\ref{introthm:asymdim-zero} (see Theorem~\ref{thm:asymdim-zero}).

  \textbf{Conventions:} we consider all maps to be left maps, and compose them left to right.
  Throughout the text, $S$ will be a quasi-countable inverse semigroup (unless otherwise specified), and its meet semilattice of idempotents shall be $\idempotents$. By a metric, we will always mean an extended metric. 
  
  \textbf{Acknowledgements:} we would like to thank Becky Armstrong for pointing out an issue with the construction of the reduced crossed product in Section~\ref{sec:roealg}.

  \section{Preliminaries} \label{sec:pre}
  
  In this section, we collect all the necessary background for the text. Given that our work lies in the intersection of operator algebras, large scale geometry and inverse semigroups, we provide an introduction to all three so that the paper is accessible to researchers with any of these backgrounds. 

\subsection{C*-algebras and approximation properties} \label{subsec:pre-roe}
We begin by giving a brief introduction to C*-algebras, focusing on the finiteness notions we use in the paper. We refer the reader to \cite{Murphy1990} for a comprehensive introduction to operator algebras; the finiteness properties can be found in~\cite{BO08}.

A \emph{Banach algebra} is an associative algebra which is a Banach space, that is, is equipped with a complete submultiplicative norm. A \emph{C*-algebra} is a Banach algebra $A$ over the field of complex numbers, together with an involution
$x \mapsto x^\ast$ such that
$$\begin{array}{ll}
	(x+y)^\ast=x^\ast+y^\ast;
	&(xy)^\ast=y^\ast x^\ast; \\
	(\lambda x)^\ast=\overline\lambda x^\ast;
	&\left\|xx^\ast\right\|=\left\|x\right\|^2
\end{array}$$
for all $x, y \in A$ and $\lambda \in \mathbb{C}$. The canonical example is the C*-algebra $\mathcal B(H)$ of all bounded linear operators on a Hilbert space $H$, where $x^\ast$ is the adjoint of $x$. Note that, in the finite dimensional case, these are just full matrix algebras. An important example for us is $\mathcal B(\ell^2(X))$, where $\ell^2(X)$ is the Hilbert space of square-summable functions from $X$ to $\mathbb C$.  In this text, we also are only concerned with C*-algebras which are unital.

Any finite-dimensional C*-algebra is isomorphic to some finite sum $M_{n_1} \oplus \dots \oplus M_{n_k}$ for some $n_1, \dots, n_k \in \mathbb{N}$, where $M_{n_i}$ denotes the full matrix algebra of size $n_i$-by-$n_i$. Within C*-algebra literature (see, e.g.,~\cite{BO08}) it is often interesting to see which C*-algebras behave similarly to finite-dimensional C*-algebras and which do not. Following Dedekind's definition of a finite set, we say that a unital C*-algebra $A$ is \emph{infinite} if it has a \emph{proper isometry}, that is, there is some $v \in A$ such that $v^*v = 1$ and $vv^* < 1$. Likewise, we say $A$ is \textit{finite} if it is not infinite. Lastly, we say that $A$ is \textit{stably finite} if $M_n(A)$ is finite for every $n \in \mathbb{N}$, where $M_n(A)$ denotes the C*-algebra formed by the $n$-by-$n$ matrices whose entries lie in $A$.

Even though stably finite C*-algebras behave in a similar manner to finite-dimensional ones, we shall need other stronger notions of finiteness. The following notions are well known, and their importance has recently been made explicit in relation to the so-called \textit{classification program} (see~\cite{TWW17} and the references therein).
\begin{definition} \label{def:qd}
	Let $A$ be a C*-algebra, and let $\Omega \subset \mathcal{B}(\mathcal{H})$ be a set of operators.
	\begin{enumerate}
		\item We say $\Omega$ is a \textit{quasi-diagonal set of operators} if for all finite $F \Subset \Omega$, $\varepsilon > 0$ and $v_1, \dots, v_k \in \mathcal{H}$ there is a finite rank orthogonal projection $p \in \mathcal{B}(\mathcal{H})$ such that $\left\|p\omega - \omega p\right\| \leq \varepsilon$ for all $\omega \in F$ and $\left\|pv_i - v_i\right\| \leq \varepsilon$ for all $i = 1, \dots, k$.
		\item We say $A$ is \textit{quasi-diagonal} if there is a faithful representation $\pi \colon A \rightarrow \mathcal{B}(\mathcal{H})$ whose image $\pi(A)$ is a quasi-diagonal set of operators.
		\item We say $A$ is \textit{strongly quasi-diagonal} if every quotient of $A$ is a quasi-diagonal C*-algebra.
	\end{enumerate}
\end{definition}

Lastly, we still need yet another form of finiteness within the C* literature. The following are the strongest notions of finiteness we have seen so far, as any AF algebra (or local AF algebra) is automatically quasi-diagonal.
\begin{definition} \label{def:af-algebra}
	Let $A$ be a C*-algebra.
	\begin{enumerate}
		\item We say $A$ is an \textit{AF algebra} (or \textit{AF} for short) if there is an increasing sequence of finite dimensional subalgebras $B_n \subset A$ with dense union, that is, $A$ is the norm-closure of $\cup_{n \in \mathbb{N}} B_n$.
		\item We say $A$ is \textit{locally finite-dimensional} (or \textit{local AF} for short) if for every $a_1, \dots, a_k \in A$ and $\varepsilon > 0$ there is a finite dimensional subalgebra $B \subset A$ and $b_1, \dots, b_k \in B$ such that $\left\|a_i - b_i\right\| \leq \varepsilon$ for every $i = 1, \dots, k$.
	\end{enumerate}
\end{definition}
AF algebras are always \emph{separable}, that is, they have a countable dense subset. However, the class of C*-algebras that are of interest to us, namely uniform Roe algebras, are almost never separable (as we discuss in the end of Subsection~\ref{subsec:pre-bddgeom}), and hence are almost never AF. The notion of local AF algebras intends to correct this handicap. Furthermore, note that, in general, AF algebras are local AF; local AF algebras are strongly quasi-diagonal; which implies quasi-diagonal. Likewise, quasi-diagonal C*-algebras are stably finite, and hence also finite. In particular, if a C*-algebra $A$ contains a non-unitary isometry $v \in A$, that is, $1 = v^*v > vv^*$, then $A$ cannot be quasi-diagonal (a proof of this fact can be found, for instance, in~\cite[Proposition~7.1.15]{BO08}). None of the reverse implications hold in general.

\subsection{Coarse geometry and uniform Roe algebras} \label{subsec:pre-bddgeom}
In this subsection we introduce the relevant notions concerning coarse geometry and uniform Roe algebras. A comprehensive introduction to the topic is~\cite{NY12}. As outlined in the introduction, we will be working with extended metric spaces rather than just metric spaces, and (unlike in the standard literature) all notions are introduced in this setting. Recall that an \emph{extended metric} $d$ on a set $X$ is a map $d \colon X \times X \rightarrow [0, \infty]$ that is reflexive, symmetric and satisfies the triangle inequality. In an extended metric space, being finite distance apart is an equivalence on the points, and the corresponding equivalence classes will be referred to as \emph{components}.

Since we deal with extended metrics throughout the paper, for convenience we will always tacitly assume all metric spaces to be extended unless otherwise stated. 

\begin{definition} \label{def:bddgeo}
	Let $(X, d)$ be a metric space.
	\begin{enumerate}
		\item \label{def:item:unifdisc} $(X, d)$ is \textit{uniformly discrete} if there is a uniform constant $c$ such that $0 < c < d(x, y)$ for every pair $x, y \in X$ with $x \neq y$.
		\item \label{def:item:bddgeom} $(X, d)$ is of \textit{bounded geometry} if it is uniformly discrete and $\sup_{x \in X} |B_r(x)| < \infty$ for every $r > 0$, where $B_r(x)=\{ y\in X:d(x,y)\leq r \}$.
	\end{enumerate}
	Likewise, we say $d$ is \textit{uniformly discrete} (resp. of \textit{bounded geometry}) when $(X, d)$ is uniformly discrete (resp. of bounded geometry).
\end{definition}
For example, finitely generated groups (or indeed, finitely generated inverse semigroups) are uniformly discrete (extended) metric spaces of bounded geometry with the word metric.

The following definition makes the idea of two metric spaces having the same large scale geometry precise. In Gromov's foundational work on the large scale geometry of finitely generated groups \cite{Gro93}, \emph{coarse equivalence} is one of the definitions of large scale equivalence considered. For finitely generated groups (and more generally, for quasi-geodesic spaces) the notion of being quasi-isometric is equivalent, but in general, quasi-isometry is strictly stronger, and too strong for the countably generated setting.
\begin{definition}
	Let $(X,d_X)$ and $(Y,d_Y)$ be metric spaces. A map $f \colon X\rightarrow Y$ is called a \textit{coarse embedding} if there are non-decreasing functions $\rho_-, \rho_+ \colon [0, \infty] \rightarrow [0, \infty]$ such that $\rho_+^{-1}(\infty)=\left\{\infty\right\}$, $\rho_-(r) \rightarrow \infty$ when $r \rightarrow \infty$ and
	$$ \rho_-\left(d_X\left(x, y\right)\right) \leq d_Y\left(f(x), f(y)\right) \leq \rho_+\left(d_X\left(x, y\right)\right) $$
	for all $x, y \in X$.
	
	If, moreover, there exists $k \geq 0$ such that the $k$-neighbourhood of $f(X)$ is $Y$, then $f$ is called a \textit{coarse equivalence}.
\end{definition}
Notice that the conditions ensure that any coarse embedding respects components, that is, $d_X(x,y)=\infty$ if and only if $d_Y(f(x), f(y))=\infty$. While it is not obvious from  the definition, coarse equivalence is in fact an equivalence relation.

Given a metric space $(X, d)$ of bounded geometry, we can define its \emph{uniform Roe algebra} in the following way (see~\cite{Roe03} for a more comprehensive discussion). If $t \in \mathcal{B}(\ell^2(X))$ is an operator, we say its \textit{propagation} is
$$ \prop\left(t\right) := \inf\left\{r \geq 0 \;\; \text{such that} \;\, \langle t\delta_x, \delta_y\rangle = 0 \,\; \text{whenever} \;\, d\left(x, y\right) > r \right\}, $$
that is, if we see $t$ as an $X$-by-$X$ matrix, then its entries are $0$ outside of a band of radius $\prop(t)$ around the diagonal. The \textit{uniform Roe algebra} of $(X, d)$, introduced by Roe in~\cite{Roe93}, is the C*-subalgebra $C_u^*(X, d)$ of $\mathcal{B}(\ell^2(X))$ generated by the operators of finite propagation, which form a non-closed *-subalgebra.

Note that we can consider $\ell^\infty(X)$, the set of all bounded functions $X \to \mathbb C$ as a subset of $\mathcal{B}(\ell^2(X))$, as these act linearly by pointwise multiplication on $\ell^2(X)$; furthermore they all have propagation $0$. Therefore we always have that $\ell^\infty(X)$ is a C*-subalgebra of $C_u^*(X, d)$ and, hence, since $\ell^\infty(X)$ is non-separable as soon as $X$ is infinite, so will $C_u^*(X, d)$ be non-separable.

Lastly, recall that $\unifroealg$ is a bijective coarse invariant of the space. Indeed, if $\varphi \colon (X, d) \rightarrow (Y, d')$ is a bijective coarse equivalence then $C_u^*(X, d) \cong C_u^*(Y, d')$. In order to prove this, note that we may define a unitary operator $u \colon \ell^2(X) \rightarrow \ell^2(Y)$ by $u\delta_x := \delta_{\varphi(x)}$. It is then routine to show that the map $a \mapsto uau^*$ defines a *-isomorphism between $C_u(X, d)$ and $C_u^*(Y, d')$.

  \subsection{Inverse semigroups} \label{subsec:pre-invsem}
  In this section, we give a quick overview of inverse semigroups. We refer the reader to \cite{L98} for a comprehensive introduction, or to \cite{lawpr} for a more concise one.
  
  A semigroup $S$ is said to be inverse if each element $s\in S$ admits a unique inverse $s^\ast \in S$ with the properties
  $$ss^\ast s=s \;\, \text{and} \;\, s^\ast s s^\ast=s^\ast.$$ 
  Groups, in particular, are inverse semigroups, but in an inverse semigroup, $ss^\ast$ and $s^\ast s$ are typically not identity elements, and are typically distinct from each other. We do however have $(s^\ast)^\ast=s$ and $(st)^\ast=t^\ast s^\ast$, which shows similarity with the $\ast$-operation in C*-algebras\footnote{In the inverse semigroup literature, the inverse is usually denoted by $^{-1}$. The $\ast$ notation is more common where inverse semigroups are used to generate C*-algebras, as in these constructions, the inverse in the semigroup does correspond to the * operation in the algebra.}. 
  Semilattices, that is, commutative semigroups where every element is an idempotent, are also inverse semigroups, with every element being its own inverse.
  
  The canonical example of an inverse semigroup is the symmetric inverse semigroup $\mathcal I_X$ on a set $X$: this consists of all bijections between subsets of $X$ (including the empty map), which forms a semigroup under the usual composition of partial maps (i.e., composing on the largest possible domain), and the inverse of a map is just the usual inverse mapping. In particular, if $\varphi \in \mathcal I_X$ then $\varphi^\ast\varphi$ is the identity map on the domain of $\varphi$, whereas $\varphi\varphi^\ast$ is the identity map on its range (we work with left maps in the paper). The common cardinality of the domain and range of $\varphi$ is called its \emph{rank}.
  
  Analogously to groups, each inverse semigroup can be represented in a symmetric inverse semigroup. The \emph{(left) Wagner-Preston representation of $S$} maps $s$ to $\lambda_s \in \mathcal I_S$, where $\lambda_s \colon s^*S \to s S, x \mapsto s x$. This also allows us to represent $S$ in the $\ast$-algebra $\mathcal B(\ell^2(S))$ of bounded linear operators on $\ell^2(S)$:
\begin{equation}
\label{eq:WPonl2}
v_s \colon \ell^2(S) \to \ell^2(S),\ v_s (f)(t)=
  \begin{cases}
  	f(\lambda_{s^*}(t))&\hbox{if }t \in sS,\\
  	0&\hbox{if } t \notin sS.
  \end{cases}
\end{equation}  
    
  An \emph{inverse monoid} is an inverse semigroup that has an identity element. For example, the symmetric inverse semigroup is also an inverse monoid with the identity map. Not all inverse semigroups are monoids; however, any inverse semigroup $S$ can be turned into an inverse monoid $S^1$ by adjoining an external identity, i.e., $S^1=S \sqcup \{1\}$ where $1$ acts as an identity element, and multiplication is otherwise inherited from $S$. This allows us to work with inverse monoids rather than just semigroups when convenient, without a loss of generality.
  
  An \emph{inverse subsemigroup} $T$ of an inverse semigroup $S$ is a subset of $S$ closed under multiplication and taking inverses. If $S$ is an inverse monoid, then $T$ is an \emph{inverse submonoid} if it is an inverse subsemigroup and it contains the identity element of $S$. For example, if $|X|=n$, then the rank $n$ maps form an inverse submonoid of $\mathcal I_X$: the symmetric group on $X$. The notions of inverse subsemigroups and inverse submonoids \emph{generated} by subsets of $S$ are defined as usual. 
  
  An element $e \in S$ is called idempotent if $e^2=e$, or equivalently, if $e^\ast=e$. The set of idempotents plays an important role in the algebraic theory of inverse semigroups, and is usually denoted by $E$. In any inverse semigroup, idempotents commute with each other and thus form an inverse subsemigroup. Since multiplication in $E$ is idempotent and commutative, it forms a semilattice, and considering multiplication as a meet operation we have the corresponding partial order given by $e \leq f$ if $ef=e$. This partial order extends naturally to the whole semigroup $S$ by defining $s \leq t$ if $ss^\ast t=s$, or equivalently, $ts^\ast s=s$, and is called the \emph{natural partial order} on $S$. An equivalent characterization is that $s \leq t$ if there exists an idempotent $e$ with $s=et$, or equivalently, if there exists an idempotent $f$ with $s=tf$.
  In the symmetric inverse semigroup $\mathcal I_X$, the idempotents are the identity maps on subsets of $X$, and the natural partial order just corresponds to restriction of the domain. An important property of the natural partial order is that it is compatible with the operations, i.e., if $s_1 \leq t_1$ and $s_2 \leq t_2$, then $s_1s_2 \leq t_1 t_2$ and $s_1 ^\ast \leq t_1^\ast$.
 
  In order to understand the structure of an inverse semigroup -- or indeed, semigroups in general -- it is key to know which elements can be multiplied into one another. This is described by the so-called \emph{Green's relations}. We say $s$ and $t$ are \emph{$\lrel$-related} if they can be mutually multiplied into each other on the left, that is, if there are elements $u,v \in S$ with $us=t$ and $vt=s$. There are several equivalent characterizations which are useful to know, such as: $s$ and $t$ generate the same \emph{left-ideal}, that is, $Ss=St$; or, equivalently, $s^\ast s=t^\ast t$. In particular, $s \lrel s^\ast s$ for any $s \in S$, and this is the unique idempotent in the $\lrel$-class of $s$, on which it acts as a left identity. It also follows that elements of the $\lrel$-class are pairwise incomparable in the natural partial order.
  
  The \emph{$\rrel$-relation} is analogously defined on the right hand side. The smallest equivalence containing both $\lrel$ and $\rrel$ is called the \emph{$\drel$-relation}. It can be proven that $s \drel t$ if and only if there exists $u \in S$ with $s \lrel u \rrel t$, or equivalently, there exists $v \in S$ with $s \rrel v \lrel t$. In a symmetric inverse semigroup $\mathcal I_X$, two maps $\varphi$ and $\psi$ are $\lrel$-related if they have the same domain, and similarly, they are $\rrel$-related if they have the same image, and $\drel$-related if they have the same rank. 
  
  We denote the $\lrel$-, $\rrel$-, and $\drel$-class of $s \in S$ by $L_s, R_s$ and $D_s$ respectively. We remark that there are two more Green's relations we have not mentioned as we will not be using them in the paper.
  


  Given an inverse semigroup $S$ generated by $X$ -- in notation, $S=\langle X \rangle$, we define the (left) Cayley graph of $S$ as the edge-labeled digraph with vertex set $S$ and edges of the form $s \xrightarrow{x} xs$ where $x \in X \cup X^\ast$, and $X^*=\{y^\ast \colon y \in X\}$.
  Cayley graphs of inverse semigroups are typically not strongly connected, indeed, note that there is a direct path $s \to t$ if and only if there exists $u\in S$ with $t=us$. This shows that the strong components are exactly the $\lrel$-classes of the inverse semigroup. The strong component of $s$ is called the \emph{Sch\"utzenberger graph} of $s$, and it follows from the work of Stephen \cite{S90} that the full Cayley graph of an inverse semigroup can be recovered from just the set of Sch\"utzenberger graphs (as edge-labeled graphs).
  Furthermore, within the  Sch\"utzenberger graphs, edges do come in inverse pairs:  $s \lrel xs$ if and only if $s=x^\ast xs$, that is, $x$ and $x^*$ are opposite edges between $s$ and $xs$.
  As first observed in \cite{GSSz} and \cite{LM21}, finitely generated inverse semigroups are therefore naturally equipped with a \emph{word metric} via the path metric within the Sch\"utzenberger graphs, by defining $d(s,t)$ to be infinite if $s \not\lrel t$, and the distance of $s$ and $t$ in their (common) Sch\"utzenberger graph if $s \lrel t$. 
  
  Unlike in the case of groups, this word metric is a meaningful metric in some non-finitely generated, even non-countable semigroups. Observe that if $e \in X$ is an idempotent and $s \lrel es$, then $s=e^*es=es$, that is, idempotents label loops within the Sch\"utzenberger graphs, and therefore they do not influence the metric. Thus, for our purposes it makes sense to consider the following generalization of generation: 

  \begin{definition} \label{def:cnt-quasi-gen}
    We say an inverse semigroup $S$ is \emph{quasi-generated} by a set $X$ if $S=\langle X \cup E \rangle$, where $E \subset S$ is the set of idempotents of $S$. We call $S$ \textit{quasi-countable} if it is quasi-generated by a countable set.
  \end{definition}
Note that this does not imply $S$ itself is countable -- any semilattice is quasi-countable by definition.

We remark that \cite{GSSz} is only concerned with finitely generated inverse semigroups, while \cite{LM21} considers the word metric in a larger class they call \emph{finitely labelable}, which in turn is a subclass of finitely quasi-generated inverse semigroups. In Section \ref{sec:metrics}, we generalize the word metric to any quasi-countable inverse semigroup.


A nice class of inverse semigroups we will often appeal to for examples are the direct products of semilattices and groups. In general, the direct product of inverse semigroups is again an inverse semigroup, so in particular, given a semilattice $E$ and a group $G$, the set $S=G \times E$ is an inverse semigroup with $(g,e)^\ast=(g^{-1}, e)$ for $g\in G$ and $e\in E$. In particular $(g,e)^\ast (g,e)=(g,e)(g,e)^\ast=(1,e)$, where $1$ is the identity element of the group. Thus here the $\lrel$, $\rrel$ and $\drel$ relations all coincide with having the same second component, and all Sch\"utzenberger graphs are isomorphic to the Cayley graph of $G$. Notice that $S$ is quasi-generated by $G$, or indeed by any generating set of $G$, so $S$ is finitely (countably) quasi-generated whenever $G$ is.

  \section{Proper and right subinvariant metrics on inverse semigroups} \label{sec:metrics}
  This section is dedicated to defining a coarse invariant metric on quasi-countable inverse semigroups. As in the finitely generated case, the metric we consider will always be a uniformly discrete, extended metric whose components are exactly the $\lrel$-classes. By a metric on an inverse semigroup, we will always mean such a metric, even when not explicitly stated. 
  
  In Subsection~\ref{subsec:metrics}, we define the adequate notions of \textit{properness} and \textit{right-invariance}. Subsections~\ref{subsec:exis-metric} and~\ref{subsec:uniq-metric} then show that such metrics always exist and, up to bijective coarse equivalence, depend on the inverse semigroup alone. Subsection~\ref{subsec:ex-metrics} shows that any uniformly discrete metric space of bounded geometry arises as a component of an inverse semigroup. Lastly, Subsection~\ref{groupoid-metric} describes the relationship between length functions on inverse semigroups and their universal groupoids.

  \subsection{Proper and right subinvariant metrics} \label{subsec:metrics}
  
    \begin{definition} \label{def:metric}
    Let $S$ be a quasi-countable inverse semigroup, and let $d \colon S \times S \rightarrow [0, \infty]$ be a metric (which is, as always, uniformly discrete and the components are the  $\lrel$-classes).
    \begin{enumerate}
      \item \label{def:metric:ri} We say $d$ is \textit{right subinvariant} if $d(sx, tx) \leq d(s, t)$ for every $s, t, x \in S$.
      \item \label{def:metric:pr} We say $d$ is \textit{proper} if for every $r \geq 0$ there is some finite set $F \Subset S$ such that $y \in Fx$ for every pair $x, y \in S$ with $x \neq y$ and $d(x, y) \leq r$.
    \end{enumerate}
  \end{definition}
  
\begin{remark}
In condition (\ref{def:metric:ri}) we may, without loss of generality, assume that $d(s,t) < \infty$, that is, $s \lrel t$, as the statement is automatic for pairs with $d(s,t)=\infty$. We will take advantage of this to shorten proofs.
\end{remark}

\begin{remark} \label{rem:left-invariance}
	There is an obvious one-sidedness to our definitions in that we chose \textit{right} subinvariance  as opposed to \emph{left} subinvariance, and $\lrel$-classes rather than $\rrel$-classes. This all stems from the fact that the variant of the Wagner-Preston representation we use to construct the crossed product algebra represents $S$ as \emph{left} maps rather than right maps, as usual in operator theory. However, this is only a convention, and the dual of the paper also holds; in fact the map $s \mapsto s^*$ is an anti-homomorphism on inverse semigroups which swaps the sides.
\end{remark}

\begin{example}[The word metric in finitely quasi-generated inverse semigroups]\label{ex:fg-word-metric}
To reassure the reader that this section is not about the empty set, we begin by observing that for finitely quasi-generated inverse semigroups, the word metric satisfies the above conditions, and thus Definition \ref{def:metric} generalizes the metric defined in \cite{GSSz} and \cite{LM21} as claimed. 
Indeed, consider the strong components of the Cayley graph with respect to the generating set $X \cup E$. These components are the $\lrel$-classes, and if $s \lrel t$, $s \neq t$, then $d(s,t)$ is the length of a minimum length path from $s$ to $t$. Since idempotents label loops, a minimum length path will only have labels coming from $X\cup X^\ast$, so
$$
	d\left(s,t\right) := \min \left\{k\mid  x_1 \ldots x_k t=s,\ x_i \in X \cup X^\ast \right\}. $$
This is right subinvariant, since $x_1 \ldots x_k t=s$ implies $x_1 \ldots x_k tx=sx$ for any $x \in S$.
To see properness, take $r> 0$, and let $F=\{x_1\ldots x_k: x_i \in X \cup X^\ast, 1\leq k \leq r\}$. Then if $d\left(s,t\right) \leq r$, then either $s=t$ or $x_1 \ldots x_k t=s$ for some $1 \leq k \leq r$ and $x_i \in X \cup X^\ast$, that is, $s \in Ft$ indeed.
\end{example}

Notice properness could fail if we did not assume $x \neq y$ in the definition, as in general there may not be closed paths with labels in $X \cup X^\ast$ around each point. Indeed, if we choose $S$ to be the semilattice $(\mathbb N, \min)$, then $S$ is quasi-generated by the empty set, but for any finite subset $F$ if $m > \max F$ then $m \notin F n$ for any $n \in \mathbb N$. While our definition of properness is similar in flavour to the finite labelability condition of \cite{LM21}, $(\mathbb N, \min)$ is not finitely labelable for this very reason.

Since the word metric is our prototype for Definition~\ref{def:metric}, we record a few of its properties here which proper, right subinvariant metrics will generalize. The following claims follow from \cite{S90}.

\begin{proposition}\label{prop:wordmetric-prop}
	Let $S$ be an inverse semigroup.
	\begin{enumerate}
		\item\label{prop:wordmetric-prop:contract} If $t \lrel st$, then the map $L_s \to L_t$, where $x \mapsto xt$, is an edge-labeled graph morphism with $s^\ast s \mapsto t$ and $s \mapsto st$. In particular if $S$ is finitely quasi-generated and $d$ is the word metric, then $d(s^\ast s,s) \geq d(t,st)$.
		\item\label{prop:wordmetric-prop:d} Any two Sch\"utzenberger graphs contained in the same $\drel$-class are isomorphic as edge-labeled digraphs.
	\end{enumerate}
\end{proposition}

Note that if $1 \not\lrel s$, then $d(1,s)=\infty$, however  $d(s^\ast s,s)$ is always finite as $s^*s \lrel s$, furthermore $s(s^*s)=s$. In light of item (\ref{prop:wordmetric-prop:contract}),  $d(s^\ast s,s)$ can be considered as the analogue of the length of $s$ in groups, and will later be introduced as such.

In the following, we present a few more examples to show why some natural, weaker alternatives to right subinvariance and properness are too weak to guarantee a coarse unique metric.

 \begin{example}[Invariance under the Wagner-Preston action is too weak] \label{ex:need-subinvariance}
	We remarked in the introduction that we cannot in general replace right subinvariance with right invariance, which most obviously fails in the case when $S$ has a $0$. However, from an inverse semigroup theory point of view, an even more natural question is whether we can replace right subinvariance with being invariant under the natural right Wagner-Preston action of $S$ on itself. 
	
	It turns out this is not strong enough to guarantee uniqueness up to bijective coarse equivalence. For a counterexample, consider the direct product $S$ of any finitely generated group $G$ and the semilattice $(\mathbb Z^-, \min)$. (This is an infinite descending chain of copies of $G$.) Consider the word metric $d_G$ in $G$, and define the following two metrics on $S$:
	
	\begin{align}
		d_1\left(\left(g, -n\right), \left(h, -m\right)\right) & :=
		\left\{
		\begin{array}{rl}
			d_G\left(g, h\right) & \text{if} \;\; n = m, \\
			\infty & \hbox{otherwise},
		\end{array}
		\right. \nonumber \\
		d_2\left(\left(g, -n\right), \left(h, -m\right)\right) & :=
		\left\{
		\begin{array}{rl}
			n \cdot d_G\left(g, h\right) & \text{if} \;\; n = m, \\
			\infty & \hbox{otherwise},
		\end{array}
		\right. \nonumber
	\end{align}

 \begin{figure}[htbp]
	\begin{center}
		\begin{tikzpicture}[scale=2]
			\foreach \i in {1,...,4}{
				\draw[thick] (0,-0.7*\i) node{$\scriptstyle\bullet$} -- node[above]{$1$} (1,-0.7*\i) node{$\scriptstyle\bullet$};
				\draw[thick] (4-\i/2,-0.7*\i) node{$\scriptstyle\bullet$} -- node[above]{$\i$} (4+\i/2,-0.7*\i) node{$\scriptstyle\bullet$};
			}
			\node at (0.5, -3.2) {$\vdots$};
			\node at (4, -3.2) {$\vdots$};
		\end{tikzpicture}
	\end{center}
	\caption{The two distances $d_1$ and $d_2$ when $G=\mathbb Z_2$ in Example \ref{ex:need-subinvariance}}
\end{figure}
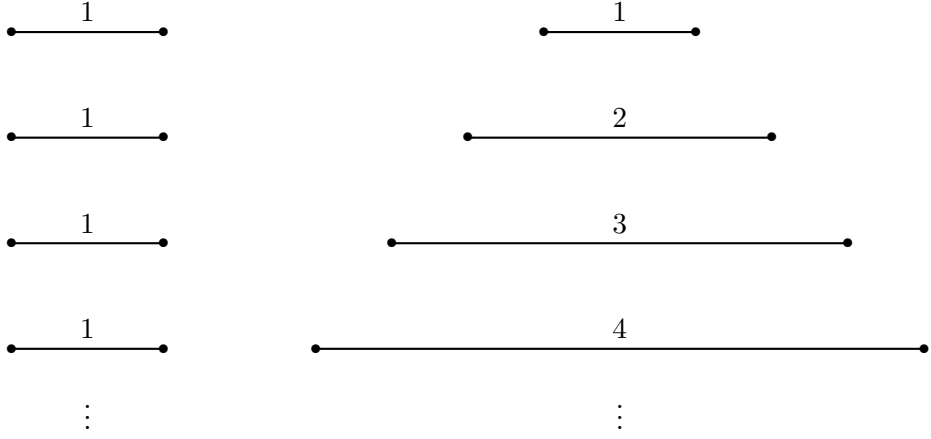

	First, notice that these two metrics are not bijectively coarse equivalent -- there is no way to bound $d_2$ by a function of $d_1$. However, both are invariant under the right Wagner-Preston action. Indeed, the action of $s\in S$ is given by $\rho_s: Ss^\ast \to Ss, x \mapsto xs$. In this case, if $s=(g,-n)$ then $Ss^\ast=\{(h,-m): -m\leq -n\}$, and $(h,-m)\mapsto (hg,\min(-n,-m))=(hg,-m)$, so the Wagner-Preston maps stabilize each group and are exactly the right Cayley maps groupwise. Clearly, both metrics are right invariant under these. (Note however that $d_2$ is not right subinvariant, while $d_1$ is.)
	Furthermore, both metrics are proper: for $r\geq 0$ one can choose $F$ to be the $r$-neighborhood of $(1,-1)$, the identity of the maximal group. Then, if $r \geq d_i\left(\left(g, -n\right), \left(h, -n\right)\right)$ for either metric, then $r\geq d_G(g,h)=d_G(1,hg^{-1})=d_i\left(\left(1, -1\right), \left(hg^{-1}, -1\right)\right)$, so $\left(hg^{-1}, -1\right) \in F$ and $(h,-n)=\left(hg^{-1}, -1\right)(g,-n)$. 
		
\end{example}

The properness condition is the more technical and perhaps more mysterious one of the two in Definition~\ref{def:metric}. As we are about to see, it is strictly stronger than having bounded geometry, and this stronger condition is indeed necessary if we want to end up with a coarse unique metric.

  \begin{lemma} \label{lemma:bddgeom}
    Let $S$ be an inverse semigroup, and let $d$ be a right subinvariant metric. If $d$ is proper then the extended metric space $(S, d)$ is of bounded geometry.
  \end{lemma}
  \begin{proof}
    Observe that if $d$ is proper then $B_r(x) \setminus \{x\} \subset F x$, where $F$ witnesses the properness of $d$ for $r$ (see Definition~\ref{def:metric}~(\ref{def:metric:pr})). Therefore $\sup_{x \in S} |B_r(x)| \leq \sup_{x \in S} |Fx| + 1 \leq |F| + 1 < \infty$, which proves that $(S, d)$ is of bounded geometry.
  \end{proof}
  
  The converse of Lemma~\ref{lemma:bddgeom} does not hold in general.

\begin{example}[Bounded geometry is too weak] \label{ex:nonfl}
 Let $G$ be any non-trivial finitely generated group, and let $S$ be the direct product of $G$ and the semilattice $(\mathbb N, \min)$. Note that $S$ is not finitely quasi-generated, as it contains an unbounded chain, but it is quasi-countable, as it itself is countable. Let $d_G$ be the word metric on the group $G$, and consider the metrics
    \begin{align}
      d_1\left(\left(g, n\right), \left(h, m\right)\right) & :=
        \left\{
          \begin{array}{rl}
            d_G\left(g, h\right) & \text{if} \;\; n = m, \\
            \infty & \textit{otherwise},
          \end{array}
        \right. \nonumber \\
      d_2\left(\left(g, n\right), \left(h, m\right)\right) & :=
      \left\{
        \begin{array}{rl}
          n \cdot d_G\left(g, h\right) & \text{if} \;\; n = m, \\
          \infty & \textit{otherwise},
        \end{array}
      \right. \nonumber
    \end{align}

 \begin{figure}[htbp]
	\begin{center}
		\begin{tikzpicture}[scale=2]
			\foreach \i in {1,...,4}{
			\draw[thick] (0,0.7*\i) node{$\scriptstyle\bullet$} -- node[above]{$1$} (1,0.7*\i) node{$\scriptstyle\bullet$};
			\draw[thick] (4-\i/2,0.7*\i) node{$\scriptstyle\bullet$} -- node[above]{$\i$} (4+\i/2,0.7*\i) node{$\scriptstyle\bullet$};
		}
			\node at (0.5, 3.5) {$\vdots$};
			\node at (4, 3.5) {$\vdots$};
		\end{tikzpicture}
	\end{center}
	\caption{The two distances $d_1$ and $d_2$ when $G=\mathbb Z_2$ in Example \ref{ex:nonfl}}
\end{figure}
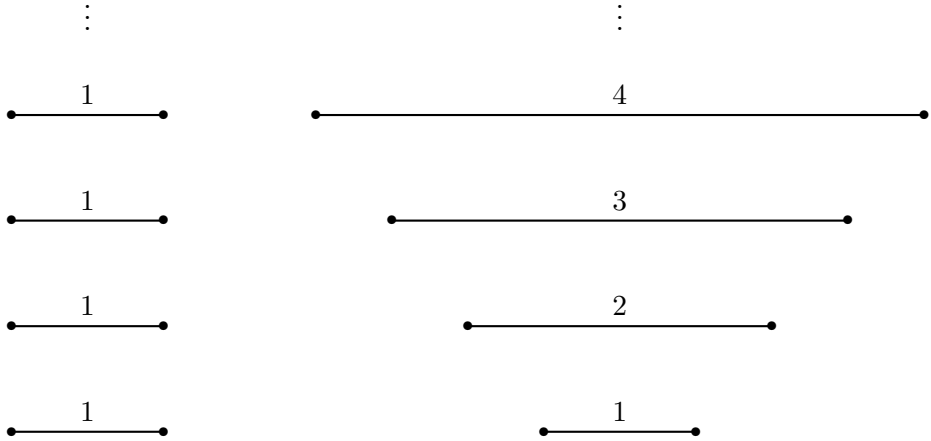

    Both metrics are right subinvariant: 
    
        \begin{eqnarray*}
    d_1\left(\left(g, n\right)(k,m), \left(h, n\right)(k,m)\right)&=&d_1\left(\left(gk, \min(m,n)\right), \left(hk, \min(m,n)\right)\right)\\
    &=&d_G(gk,hk)=d_G(g,h)= d_1\left(\left(g, n\right), \left(h, n\right)\right)
    \end{eqnarray*}
    
    \begin{eqnarray*}
    d_2\left(\left(g, n\right)(k,m), \left(h, n\right)(k,m)\right)&=&d_2\left(\left(gk, \min(m,n)\right), \left(hk, \min(m,n)\right)\right)\\
    &=&\min(m,n) \cdot d_G(gk,hk)=\min(m,n) \cdot d_G(g,h)\\
    &\leq & n \cdot d_G(g,h)= d_2\left(\left(g, n\right), \left(h, n\right)\right)
    \end{eqnarray*}

    Also, both metrics are of bounded geometry, however similarly to Example~\ref{ex:need-subinvariance} they are not coarsely equivalent as $d_2$ cannot be bounded by a function of $d_1$.
    However, $d_1$ is not proper, while $d_2$ is. Indeed, observe any finite set $F \Subset S$ is contained in $G \times \{1, \ldots, n\}$ for some $n$, and the set $G \times \{1, \ldots, n\}$ is an ideal, so we have $Fs \subseteq G \times \{1, \ldots, n\}$ for any $s \in S$. This shows that $d_1$ cannot be proper. However in the case of $d_2$, the ball $B_r(g,n)$ only contains elements other than $(g,n)$ if $n \leq r$, so choosing $F=B_r(1,r)$ will suffice by the same argument as seen in Example \ref{ex:need-subinvariance}. 
  \end{example}  
  
  For groups, properness in the sense of Definition~\ref{def:metric} is equivalent to uniformly bounded geometry, and thus properness in the group sense: indeed, given $r \geq 0$, taking $F=B_r(1)$ where $1$ is the identity element of the group will suffice. Right subinvariance is also equivalent to right invariance given $d(g,h)\geq d(gx,hx) \geq d(gxx^{-1}, hxx^{-1})$ implies $d(g,h)= d(gx,hx)$. This shows that Definition~\ref{def:metric} does indeed extend the definition used for groups.

  The following lemma contains a few simple observations about the components of the metric.
  
  \begin{lemma} \label{lemma:resp-lclasses}
    Let $d$ be a metric on $S$ with $\lrel$-classes as components. Then the following hold:
    \begin{enumerate}
      \item \label{lemma:resp-lclasses:group} $S$ is a group if and only if $(S, d)$ is a non-extended metric space. This, in turn, is equivalent to $S$ having precisely one idempotent.
      \item \label{lemma:resp-lclasses:lattice} $S$ is a semilattice if and only if $d(x, y) = \infty$ whenever $x, y \in S$ and $x \neq y$.
      \item \label{lemma:resp-lclasses:ce-latt} $(S, d)$ is coarsely equivalent to $(E, d)$ if and only if each component of $S$ has diameter at most $c$ for some $c\geq 0$.
    \end{enumerate}
  \end{lemma}
  \begin{proof}

    The first statement~(\ref{lemma:resp-lclasses:group}) follows from the fact that $S$ is a group if and only if $\lrel$ is the universal relation.

    Condition~(\ref{lemma:resp-lclasses:lattice}) follows from the observation that $S$ is a semilattice if and only if each element of $S$ is idempotent, that is, each $\lrel$-class is a singleton.

    Lastly, for~(\ref{lemma:resp-lclasses:ce-latt}), notice that since coarse equivalences preserve components, any coarse embedding from $(E,d)$ must map different elements to different $\lrel$-classes. Furthermore, since any permutation of $E$ is an isometry of $(E,d)$, a coarse equivalence exists from $(E,d)$ to $(S,d)$ if and only if the inclusion map $\iota  \colon E \rightarrow S$ is a coarse equivalence. This, in turn, is equivalent to $\iota$ being coarse-surjective, i.e., there is some constant $k \geq 0$ such that for every $s \in S$ there is some $e \in E$ with $d(e, s) \leq k$. Since $d(e, s) \leq k$ in particular implies $e \lrel s$, we must have $e=s^\ast s$ here, the unique idempotent in the $\lrel$-class. Thus taking $c=2k$ proves the statement.
  \end{proof}
  The following lemma will allow us to assume $S$ is a monoid when it is convenient to do so.
  \begin{lemma} \label{lemma:unital-technicality}
    Let $S$ be an inverse semigroup, and $d$ be a metric. Moreover, let $S^1 := S \sqcup \{1\}$ be $S$ with an added identity. Then there is exactly one metric $d^1$ on $S^1$ that extends $d$.
  \end{lemma}
  \begin{proof}
    Since by construction, $1$ does not arise as a product with a factor in $S$, it follows that $1$ forms a trivial $\drel$-class, and hence
    $$ d^1(s, t) =
        \left\{
          \begin{array}{rl}
            d(s, t) & \text{if} \; s, t \in S, \\
            0 & \text{if} \; s = t = 1 \in S^1, \\
            \infty & \text{otherwise}.
          \end{array}
        \right. $$ 
  \end{proof}

The next lemma shows that the analogue of Proposition~\ref{prop:wordmetric-prop} holds for any metric satisfying the conditions of Definition \ref{def:metric}.

  \begin{lemma} \label{lemma:metric-right-inv}
    Let $S$ be a quasi-countable inverse semigroup, and let $d$ be a right subinvariant metric. Then the following assertions hold:
      \begin{enumerate}
    	\item\label{lemma:metric-right-inv:contract} If $t \lrel st$, then $d(t,st) \leq d(s^\ast s,s)$.
    	\item \label{lemma:metric-right-inv:po} If $y \leq x$, then $d(y^\ast y, y) \leq d(x^\ast x, x)$.
       \item \label{lemma:metric-right-inv:iso} The map $L_{s} \rightarrow L_{s^*}$, where $t \mapsto ts^*$, is $d$-isometric, where $L_{s}$ and $L_{s^*}$ are the $\lrel$-classes of $s$ and $s^*$ respectively.
      \item \label{lemma:metric-right-inv:eq} $d(s, t) = d(ss^*, ts^*)$ for every pair $s, t \in S$ such that $s \lrel t$.
      \item \label{lemma:metric-right-inv:q} Given any two $\lrel$-classes in the same $\drel$-class, there is an isometry between them.
      
    \end{enumerate} 
  \end{lemma}
  \begin{proof}
  	For~(\ref{lemma:metric-right-inv:contract}), recall that $t \lrel st$ implies $t=s^*st$, so $d(s^\ast s,s) \geq d(s^\ast st,st)=d(t,st)$ by right subinvariance.
  	Item (\ref{lemma:metric-right-inv:po}) is an immediate consequence: if $x \geq y$ then $xy^*y=y$, so applying~(\ref{lemma:metric-right-inv:contract}) with $s=x$ and $t=y^*y$ we get $d(x^*x,x)\geq d(y^*y, xy^*y) =d(y^*y, y)$ indeed.  	
  	
For~(\ref{lemma:metric-right-inv:iso}) take $t_1, t_2 \in L_s$, that is, $t_1^*t_1=t_2^*t_2=s^*s$. First note that 
$$(t_is^*)^*t_is^*=st_i^*t_is^*=ss^*=(s^*)^*s^*,$$ so $t_is^* \in L_{s^\ast}$
indeed. By right subinvariance 
$$d(t_1,t_2)=d(t_1s^*s, t_2s^*s)\leq d(t_1s^*, t_2s^*) \leq d(t_1,t_2)$$ which proves the claim. Notice that~(\ref{lemma:metric-right-inv:eq}) immediately follows from~(\ref{lemma:metric-right-inv:iso}). For~(\ref{lemma:metric-right-inv:q}), consider two $\drel$-related $\lrel$-classes $L_1$ and $L_2$, and let $s_i \in L_i$.
As $s_1 \drel s_2$, there exists $t \in S$ such that $s_1 \lrel t \rrel s_2$, that is, $L_t=L_1$ and $t^* \lrel s_2^\ast$, so $L_{2}=L_{t^*}$. Thus $L_1$ and $L_2$ are isometric by~(\ref{lemma:metric-right-inv:iso}). 
  \end{proof}
  
  \subsection{Existence of metrics} \label{subsec:exis-metric}
  In this subsection we show that every quasi-countable inverse semigroup admits an essentially unique, proper, and right subinvariant metric. The first part, Proposition~\ref{prop:metric-existence}, proves the existence of such metrics, while the second one, Proposition~\ref{prop:metric-unique}, proves these metrics are unique up to coarse equivalence. The proofs of both of these facts are inspired by those for countable groups (see, e.g.,~\cite{NY12}).

As in the case of groups, metrics correspond naturally to \textit{length functions}, and it is often more convenient to define these rather than metrics. Proposition \ref{prop:length-functions-and-metrics} below establishes the correspondence between the two.

  \begin{definition} \label{def:chains}
    Let $(P, \leq)$ be a partially ordered set and $A \subset P$. We say $A$ is \textit{finitely upper bounded} (\textit{f.u.b.} for short) if there is some finite $F \Subset P$ such that for every $a \in A$ there is some $b \in F$ with $a \leq b$.
  \end{definition}

\begin{definition} \label{def:lengthfunction}
Let $S$ be a quasi-countable inverse semigroup, and let $l\colon S \to [0, \infty)$. We say that $l$ is a \textit{length function} if $\inf_{s \in S \smallsetminus E} l(s) > 0$ and for any $s,t \in S$ the following hold:
\begin{enumerate}
\item \label{item:length-zero} $l(s)=0$ if and only if $s \in E$;
\item \label{item:length-star} $l(s)=l(s^*)$;
\item \label{item:length-product} $l(st) \leq l(s)+l(t)$. 
\end{enumerate}
Moreover, we say that $l$ is \textit{proper} if for any $r \in \mathbb R^+$, the set $C_r=\{s \in S: 0< l(s) \leq r\}$ is finitely upper bounded.
\end{definition}

\begin{remark}
  The condition that $\inf_{s \in S \smallsetminus E} l(s) > 0$ guarantees that the metric associated to $l$ in Proposition~\ref{prop:length-functions-and-metrics} is uniformly discrete, albeit this condition is not actually needed from a purely coarse geometric point of view.
\end{remark}

For the purposes of this paper, the set $C_r$ plays the role of the $r$-ball of a discrete countable group. For this reason, we shall call $C_r$ the \textit{$r$-cylinder} of $S$. 

Note that any length function $l$ on $S$ respects the partial order of $S$, in the sense that $l(t)\leq l(s)$ whenever $t \leq s$, as in this case $l(t)=l(st^*t)\leq l(s)+l(t^*t)=l(s)$.

\begin{proposition} \label{prop:length-functions-and-metrics}
If $d$ is a right subinvariant metric on $S$, then 
$$l(s)=d(s^\ast s,s)$$
defines a length function on $S$. Conversely, if $l \colon S \to [0, \infty)$ is a length function on $S$, then 
$$d(s,t) =
\left\{\begin{array}{rl}
l(ts^\ast)&\hbox{if } s \lrel t,\\
\infty&\hbox{otherwise}
\end{array}\right.$$
is a right subinvariant metric on $S$.

These define mutually inverse mappings between the length functions and right subinvariant metrics, and proper length functions correspond exactly to proper metrics.
\end{proposition}

\begin{proof}
We begin by the first statement, that is, we need to show that $l(s)$ satisfies Definition \ref{def:lengthfunction}. Indeed, for (\ref{item:length-zero}), notice that $d(s^*s,s)=0$ if and only if $s=s^*s$, which holds if and only if $s \in E$. For (\ref{item:length-star}), it suffices to show that $d(s^*s,s)=d(s^*, ss^*)$, which follows from Lemma \ref{lemma:metric-right-inv} (\ref{lemma:metric-right-inv:iso}). For property (\ref{item:length-product}), 
$$l(st)=d(t^*s^*st, st) \leq d(t^*s^*st, tt^*s^*st)+d(tt^*s^*st, st)$$ by the triangle inequality, and 
$$d(t^*s^*st, tt^*s^*st)=d(t^*t(t^*s^*st), t(t^*s^*st)) \leq d(t^*t,t)=l(t)$$ by right subinvariance. Furthermore, as $tt^*s^*st=s^*stt^*t=s^*st$, we have 
$$d(tt^*s^*st, st)=d(s^*st, st)\leq d(s^*s,s)=l(s)$$ by right subinvariance again, proving the claim.

Lastly, we show that when $d$ is proper, so is $l$. Indeed, there exists a finite set $F$ such that for any $s \in C_r$, we have $s \in Fs^*s \subseteq FE$, so $C_r$ is upper bounded by $F$.

For the converse statement, we first prove that $d$ is a metric. For any $s,t \in S$, by definition we have $d(s,t)=0$ if and only if $s^*s=t^*t$ and $ts^* \in E$. If $s=t$, then this clearly holds. Conversely, if  $ts^* \in E$, then $(ts^*)^*=st^* \in E$ as well, so 
$$s=ss^*s =st^*t \leq t = tt^*t=ts^*s \leq s,$$
hence $s=t$.
Symmetry of $d$ immediately follows from Definition \ref{def:lengthfunction} (\ref{item:length-star}). For the triangle inequality, let, $s,t,u \in S$. Notice that if $s \not\lrel u$ or $t\not\lrel u$, then $d(s,t) \leq d(s,u)+d(u,t)$ is immediate, and whence we may assume $s^*s=u^*u=t^*t$. Then
$$d(s,t)=l(ts^*)=l(tu^*us^*)\leq l(tu^*)+l(us^*)=d(u,t)+d(s,u)$$
by Definition \ref{def:lengthfunction} (\ref{item:length-product}). Right subinvariance of $d$ follows automatically, for if $s,t,x \in S$ then 
$$d(sx,tx)=l(txx^*s^*)\leq l(ts^*)=d(s,t).$$
Finally, we show that if $l$ is proper, so is $d$. Assume $C_r$ is finitely upper bounded by $F$, and let $0<d(s,t)=l(ts^\ast) \leq r$. Then $ts^\ast \leq m$ for some $m \in F$, and as $s^*s=t^*t$ that means $ts^*=m(st^*)(ts^*)=mss^*$, so $ms=mss^*s=ts^*s=tt^*t=t$, that is $t \in Fs$.

Lastly, to prove that the above maps are mutually inverse, assume that $s \lrel t$ and note that
$$l(s)\stackrel{def}{=}d(s^*s,s)\stackrel{def}{=}l(ss^*s)=l(s)$$ and
$$d(s,t)\stackrel{def}{=}l(t^*s)\stackrel{def}{=}d(st^*ts^*, ts^*)=d(ss^*,ts^*)=d(s,t).$$
\end{proof}

The following lemma is a useful way to rephrase quasi-countability.

  \begin{lemma} \label{lemma:cnt-quasi-gen:subsem}
    Let $S$ be a quasi-countable inverse semigroup. Then there is a countable inverse subsemigroup $T \subset S$ such that for every $s \in S$ either $s \in E$ or there is some $t \in T$ and $e \in E$ with $s = te$.
  \end{lemma}
  \begin{proof}
    Let $F \subset S$ be a countable set such that $S = \langle F \cup E\rangle$, and let $T := \langle F\rangle$. Given any $s \in S \setminus E$ there are $t_1, \dots, t_n \in T$ and $e_1, \dots, e_n \in E$ such that $s = t_1e_1 \dots t_ne_n$, hence  either $s \in E$ or $s \leq t_1 \ldots t_n$, and so $s=t_1 \ldots t_n s^\ast s$.
  \end{proof}

  \begin{proposition} \label{prop:metric-existence}
    Every quasi-countable inverse semigroup $S$ admits a proper, right subinvariant metric.
  \end{proposition}
  \begin{proof}
    We define the metric via a proper length function. Since $S$ is quasi-countable let $T \subset S$ be as in Lemma~\ref{lemma:cnt-quasi-gen:subsem}, that is, $T$ is a countable subsemigroup such that $S = TE \cup E$. Moreover, choose an ascending sequence of finite symmetric subsets $T_1 \subset T_2 \subset \dots \subset T$ such that $\cup_{n \in \mathbb{N}} T_n = T$. Furthermore, and without loss of generality, we may assume that $T_n T_m \subset T_{n + m}$ just by adding to $T_n$ all the products of the form $st$ where $s \in T_i$ and $t \in T_{n-i}$ for $i = 1, \dots, n-1$. Consider then the subsets
    $$\ol C_0 := E, \;\; \text{and} \;\; \ol C_n := E \cup T_n E = E \cup \left\{s \in S \mid s \leq t \; \text{for some} \; t \in T_n \right\}. $$
Then we can define a length function $l\colon S \to [0, \infty)$ by putting $l(s)=n$ if $s \in \ol C_n \setminus \ol C_{n-1}$. 

We need to show that this function satisfies all the conditions in Definition \ref{def:lengthfunction}. Property~(\ref{item:length-zero}) is immediate from the definition, while (\ref{item:length-star}) follows from the symmetry of the subsets $T_i$, and (\ref{item:length-product}) holds too as $T_n T_m \subset T_{n + m}$ implies $\ol C_n \ol C_m \subseteq \ol C_{n+m}$. Lastly, $l$ is proper since for any $r \in \mathbb Z^+$, $T_r$ is a finite upper bound for $C_r=\ol C_r \setminus \ol C_{r-1}$.
  \end{proof}

A more hands-on way to construct the metric is by generalizing the word metric in finitely quasi-generated inverse semigroups. The classical word metric with an infinite quasi-generating set generally yields a non-proper (in fact, non-uniformly bounded) metric, as the vertices in the Sch\"utzenberger graphs can have infinitely many neighbours. However, we can rectify this by making some edges longer in the metric than length $1$. 

\begin{definition} \label{def:weigthed-metric}
	Let $S = \langle X \cup E \rangle$ for a countable, symmetric set $X$, and let $w \colon X \to \mathbb N$ be a function such that $w(x)=w(x^*)$ and $w^{-1}(n)$ is finite for any $n \in \mathbb N$. The \emph{weighted word metric} on $S$ is defined by taking the path metric on the Sch\"utzenberger graphs where an edge labeled by $x$ has length $w(x)$. Formally, if $s \lrel t$ with $s \neq t$, then 
	$$ d\left(s,t\right) := \min \left\{w(x_1)+ \ldots +w(x_k)\mid x_1 \ldots x_k t=s, x_i \in X \cup X^\ast  \right\}. $$
\end{definition}

\begin{proposition} \label{prop:weigthed-metric}
	The weighted word metric is a proper, right subinvariant metric on quasi-countable inverse semigroups.
\end{proposition}
\begin{proof}
	The proof is completely analogous to the finitely quasi-generated case in Example \ref{ex:fg-word-metric}. The components of the metric are, by definition, the $\lrel$-classes, and right subinvariance can be seen in exactly the same way as before.
	For properness, given any $r \geq 0$ let $$F=\{x_1\ldots x_k\mid x_i \in X \cup X^\ast, w(x_1)+\ldots+w(x_k)\leq r\}.$$ Notice that the condition $w^{-1}(n)$ is finite ensures that $F$ is finite. Suppose $d(s,t) \leq r$, that is, either $s=t$ or $x_k \ldots x_1 t=s$ with $w(x_1)+ \ldots +w(x_k) \leq r, x_i \in X \cup X^\ast$, in which case $s \in Ft$ indeed.	
\end{proof}

Observe that in Example \ref{ex:nonfl}, the metric $d_2$ was a weighted word metric. Indeed, if the group $G$ is generated by the finite set $X$, then $S$ is generated by $X \times \mathbb N$, a countable set. The function $w\colon X \times \mathbb N \to \mathbb N$, $(x,n) \mapsto n$ yields exactly the metric $d_2$.

\subsection{Uniqueness of metrics.} \label{subsec:uniq-metric}
We proceed by showing that right subinvariant, proper metrics on inverse semigroups are unique up to bijective coarse equivalence.

  \begin{lemma} \label{lemma:unbounded-values-grow}
    Let $S$ be an inverse monoid equipped with a proper length function $l$, and let $A \subseteq S$. Then $A$ is finitely upper bounded if and only if $\sup_{s \in A} l(s) < \infty$.
  \end{lemma}
  \begin{proof}
    First suppose that for any $s \in A$, $l(s) \leq r$ for some $r$. Then, because the cylinder $C_r$ is finitely upper bounded, there is a finite upper bound $F$ for $A \setminus E$. Since $S$ has an identity, adding that identity to $F$ ensures it is a finite upper bound for $A$.
    
    For the converse, assume that $\sup_{s \in A} l(s) = \infty$, and $A$ is upper bounded by some set $F$.  Take a sequence $\{s_n\}_{n \in \mathbb{N}} \subset A$ such that $l(s_n) \geq n$ for every $n$. If $s_n \leq t$ for some $t \in F$, then $l(s_n) \leq l(t)$, so
    $$ \infty = \sup_{n \in \mathbb{N}} l(s_n) \leq \sup_{t \in F} l(t), $$
    hence $F$ cannot be finite.
  \end{proof}

This gives us the following equivalent characterizations of \emph{coarsely trivial} inverse semigroups:

\begin{corollary} \label{cor:coarse-finite}
For any inverse monoid $S$ equipped with a proper right subinvariant metric $d$, the following are equivalent:
\begin{enumerate}
	\item \label{cor:coarse-finite:E} $(S,d)$ is coarsely equivalent to $(E,d)$.
	\item \label{cor:coarse-finite:findiam} $\sup_{x \lrel y} d(x,y)< \infty$,
	\item \label{cor:coarse-finite:finlength} $\sup_{s \in S} d(s^\ast s,s)< \infty$,
	\item \label{cor:coarse-finite:finbound} $S$ is finitely upper bounded.
\end{enumerate}
\end{corollary}

\begin{proof}
	The equivalence of (\ref{cor:coarse-finite:E}) and (\ref{cor:coarse-finite:finlength}) follows directly from Lemma~\ref{lemma:resp-lclasses}~(\ref{lemma:resp-lclasses:ce-latt}).
	The equivalence of (\ref{cor:coarse-finite:finlength}) and (\ref{cor:coarse-finite:finbound}) is immediate from Lemma~\ref{lemma:unbounded-values-grow} and Proposition~\ref{prop:length-functions-and-metrics}, and clearly (\ref{cor:coarse-finite:findiam}) implies (\ref{cor:coarse-finite:finlength}). For the converse, recall that by Proposition~\ref{prop:length-functions-and-metrics}, if $x \lrel y$ then for $s=yx^*$, we have $d(x,y)=l(s)=d(s^*s,s)$.
\end{proof}

We are now ready to prove coarse uniqueness of the metric.

  \begin{proposition} \label{prop:metric-unique}
    Let $S$ be a quasi-countable inverse semigroup, and let $d, d'$ be two proper and right subinvariant  metrics on $S$. The identity map $\text{id} \colon (S, d) \rightarrow (S, d')$ is then a bijective coarse equivalence.
  \end{proposition}
  \begin{proof}
    By Lemma~\ref{lemma:unital-technicality} we may suppose, without loss of generality, that $S$ is a monoid.

	\textbf{Case 1.} $S$ is finitely upper bounded, that is, $\sup_{x \lrel y} d(x,y), \sup_{x \lrel y} d'(x, y) < \infty$ by Corollary~\ref{cor:coarse-finite}. In this case, both metrics are coarsely trivial, hence the identity map is a coarse equivalence.
	
	\textbf{Case 2.} $S$ is not finitely upper bounded, that is, $\sup_{x \lrel y} d(x,y)= \sup_{x \lrel y} d'(x, y) = \infty$ by Corollary~\ref{cor:coarse-finite}.
	
    Clearly, it is enough to prove that $\text{id} \colon (S, d) \rightarrow (S, d')$ is a coarse embedding, as it is already a bijection. That is, it suffices to construct $\rho_-, \rho_+ \colon [0, \infty] \rightarrow [0, \infty]$ non-decreasing functions such that $\rho_-(r) \rightarrow \infty$ when $r \rightarrow \infty$, $\rho_+^{-1}(\infty) = \infty$ and
    \begin{equation} \label{eq:ce-one}
      \rho_-\left(d\left(x, y\right)\right) \leq d'\left(x, y\right) \leq \rho_+\left(d\left(x, y\right)\right).
    \end{equation}
	
	  Notice that Eq.~(\ref{eq:ce-one}) is automatic when $x \not\lrel y$, and if $x \lrel y$ then
letting $s:=yx^*$, by Proposition~\ref{prop:length-functions-and-metrics} we have $d(x,y)=d(s^*s,s)$ and $d'(x,y)=d'(s^*s,s)$, so it follows it is enough to show that 
	\begin{equation} \label{eq:ce-two}
		\rho_-\left(d\left(s^*s, s\right)\right) \leq d'\left(s^*s, s\right) \leq \rho_+\left(d\left(s^*s, s\right)\right)
	\end{equation}
	for every $s \in S$. It is clear that the functions
	$$ \rho_-\left(r\right) := \min \left\{d'\left(s^*s, s\right) \mid d\left(s^*s, s\right) \geq r\right\} \;\; \text{and} \;\; \rho_+\left(r\right) := \sup \left\{d'\left(s^*s, s\right) \mid d\left(s^*s, s\right) \leq r\right\} $$
	satisfy Eq.~(\ref{eq:ce-two}) and are non-decreasing (note that the minimum in the definition of $\rho_-$ exists as $d'$ is uniformly discrete). Therefore all that is left to prove is that $\rho_-(r) \rightarrow \infty$ when $r \rightarrow \infty$ and that $\rho_+(r) < \infty$ for every $r \in [0, \infty)$. 
	
	It is immediate from the definition that $\rho_+(0)=0$. Since $d$ is proper and $S$ has an identity, 
	$\{s \in S: d(s^*s,s) \leq r\}$ is finitely upper bounded, so $\sup \left\{d'\left(s^*s, s\right) \mid d\left(s^*s, s\right) \leq r\right\}  < \infty$ indeed by Lemma~\ref{lemma:unbounded-values-grow}.
	
    To show that $\rho_-(r) \rightarrow \infty$ when $r$ grows, assume for contradiction that $\rho_-(r) \leq m$ for all $r \in [0,\infty)$.    
    Then, by the definition of $\rho_-$, for each $n \in \mathbb{N}$ there is $s_n \in S$ such that $d'(s_n^*s_n, s_n) \leq m$ and $d(s_n^*s_n, s_n) \geq n$. However, by the properness of $d'$ and the fact that $S$ is a monoid,
$\{s_n: n \in \mathbb{N}\}$ is finitely upper bounded, so by Lemma~\ref{lemma:unbounded-values-grow}, $\{d(s^*_ns_n, s_n)\}_{n \in \mathbb{N}}$ is a bounded sequence, which contradicts the assumption.
  \end{proof}

  The preceding results immediately give the following theorem.
  \begin{theorem} \label{thm:metric}
    Let $S$ be an inverse semigroup. Then the following statements are equivalent:
    \begin{enumerate}
      \item \label{thm:metric:genera} $S$ is quasi-countable;
      \item \label{thm:metric:metric} $S$ admits a proper and right subinvariant uniformly discrete metric whose components are the $\lrel$-classes. 
    \end{enumerate}
    Moreover, such a metric is unique up to bijective coarse equivalence.
  \end{theorem}
  \begin{proof}
    If $S$ is quasi-countable, then the existence of such a metric is provided by Proposition~\ref{prop:metric-existence}, while its uniqueness is given by Proposition~\ref{prop:metric-unique}. For the converse, given a proper metric $d$ observe we can construct finite subsets $F_n \Subset S$ such that if $s \in S\setminus E$ and $d(s,s^*s)\leq n$, then $s \in F_n s^*s$. Then the set $F := \cup_{n \in \mathbb{N}} F_n$ is countable and $F \cup E$ generates $S$.
  \end{proof}
 Henceforth, we shall implicitly assume that any quasi-countable inverse semigroup $S$ is canonically equipped with a proper and right subinvariant metric, usually denoted by $d$.

  \subsection{Metrics arising from inverse semigroups} \label{subsec:ex-metrics}
 
By extending the study of countable groups as geometric objects to inverse semigroups, we gain algebraic tools to study a much bigger array of metric spaces than before. Indeed, in this subsection we show that any uniformly discrete metric space $X$ of bounded geometry arises as the $\mathcal L$-class of some countable inverse semigroup $S$, up to bijective coarse equivalence.

The idea behind the proof is very straightforward. Let $X$ be a uniformly discrete metric space of bounded geometry  (note this implies $X$ is countable). For $x,y \in X$, let $\gamma_{y,x}$ be the map with domain $\{x\}$ and range $\{y\}$, mapping $x \mapsto y$. Notice that $\gamma_{y,x}$ sits inside the symmetric inverse monoid $\sym X$ as it is a partial injective map.

We will define $S$ as a countably generated inverse subsemigroup of $\sym X$ containing all the rank $1$ maps $\gamma_{y,x}$ for any $x,y \in X$. Recall from Subsection \ref{subsec:pre-invsem} that $\gamma_{y,x} \lrel \gamma_{v,u}$ if and only if $x=u$, so the rank $1$ $\lrel$-classes are in bijection with $X$.
Therefore, fixing any $x \in X$, the metric space $X$ is in natural bijection with the $\lrel$-class $L_x$ corresponding to the domain $\{x\}$ via the map $\gamma_{y,x} \mapsto y$. The idea is to define a proper, right subinvariant metric $d$ on $S$ such that $d(\gamma_{y,x}, \gamma_{z,x})$ is uniformly close to $d_X(y,z)$, yielding that $L_x$ is in bijective coarse equivalence with $X$.

One's first idea might be to just take the inverse subsemigroup of $\sym X$ consisting of all rank $1$ maps and the empty map, with the distance inherited from $X$. The issue is that this metric may not be proper. Properness would mean that if $0 < d(\gamma_{x,y}, \gamma_{x,z}) =d(y,z) \leq n$, then $\gamma_{y,x} \in F\gamma_{z,x}$ for some finite set $F$ -- observe that $\gamma_{y,x} \in F\gamma_{z,x}$  holds if and only if there is a map $\varphi \in F$ with $z \mapsto y$. If there are infinitely many pairs of points with distance at most $n$, there is no such a finite subset in $\{\gamma_{y,x} \colon x,y \in X\}$. We rectify this by adding in higher rank maps from $\sym X$ -- selecting these is where the proof becomes technical. We use some ideas from combinatorics.

Let $\Gamma$ be any (undirected, simple, possibly infinite) graph. Recall that a \emph{matching} in $\Gamma$ is a set of edges with no common endpoints.  Any matching $M$ determines a permutation $\varphi_M$ on the vertices which swaps the endpoints of each edge in the matching and fixes any vertex not covered by the matching.
Vizing's theorem states if $\Gamma$ is an undirected simple (possibly infinite) graph where the degree of any vertex is at most $\Delta$, then its edges can be partitioned into at most $\Delta+1$ matchings.\footnote{Vizing's theorem states this in the language of edge-colorings -- the parts in the partition correspond to edges sharing the same color, which are always matchings.}

For any $n \in \mathbb N$, consider the undirected graph $\Gamma_n$ with vertex set $X$ and an edge between $y$ and $z$ whenever $n-1<d_X(y,z)\leq n$. Note that since $X$ has bounded geometry, $\Gamma_n$ has uniformly bounded degree. Therefore, by Vizing's theorem, there is a finite family $\mathcal M_n=\{M_{n,j}: j \in J_n\}$ of matchings containing all its edges. Consider the bijective maps $\varphi_{n,j}=\varphi_{M_{n,j}}$ defined by the matchings, and put $F_n=\{\varphi_{n,j}: j \in J_n\} \subseteq \sym X$. Notice that by construction we have 
\begin{equation}
	\label{eqn:Fnbound}
n-1 < d_X(x,\varphi(x)) \leq n \hbox{ for any } x \in X, \varphi \in F_n, \varphi(x)\neq x.
\end{equation}

Let $S$ be the inverse submonoid of $\sym X$ generated by the sets $F_n$, where $n \in \mathbb N$ and the identity maps on singletons $\id_{y}, y \in X$. Notice that for any pair of points $x,y \in X$ such that $n-1<d(x,y)\leq n$, there is a map $\varphi$ in $F_n$ with $\varphi(x)=y$, and for this map, $\varphi \id_x=\gamma_{y,x}$. This shows that all maps of rank $1$ are in $S$.

The inverse semigroup will in fact consist of $3$ $\drel$-classes: a $0$ element at the bottom (the empty map), the $\drel$-classes of the rank $1$ maps in the middle, and the group generated by the bijections $\{\varphi_{n,j}\}$ on top.

Equip $S$ with the right subinvariant, proper metric induced by assigning weights to the (non-idempotent) generators:
$w(\varphi)=n \hbox{ for } \varphi \in F_n.$
We then, by Eq.~(\ref{eqn:Fnbound}), have
\begin{equation}
	\label{eqn:Fnweightbound}
	w(\varphi)-1 < d_X(x,\varphi(x)) \leq w(\varphi) \hbox{ for any } x \in X, \varphi \in F=\bigcup_{n \in \mathbb{N}} F_n.
\end{equation}

\begin{theorem}
	\label{thm:wobbling-inv-sem}
	The inverse semigroup $S$ contains an $\lrel$-class which is bijectively coarse equivalent to the metric space $X$.
\end{theorem}

\begin{proof}
We show that the $\lrel$-class $L_x$ of all maps with domain $\{x\}$ is such an $\lrel$-class for any $x \in X$. As observed, there is a natural bijection
$\Phi: X \to L_x$, $y \mapsto \gamma_{y,x}$ between $L_x$ and $X$. We now prove that this is a coarse equivalence. 

Let $y, z \in X$ be arbitrary and put $\gamma_{y,x}=\varphi_y \id_x$, $\gamma_{z,x}=\varphi_z \id_x$ for some $\varphi_y, \varphi_z \in F$. Recall that 
$$d(\gamma_{y,x}, \gamma_{z,x})=\min\left\{\sum_{i=1}^n w(\varphi_i): \varphi_i \in F, \varphi_n \cdots \varphi_1\gamma_{y,x}=\gamma_{z,x}\right\}$$
by definition, and $\varphi_n \cdots \varphi_1\gamma_{y,x}=\gamma_{z,x}$ is equivalent to $\varphi_n \cdots \varphi_1(y)=z$. As there is a map $\varphi \in F$ with $\varphi(y)=z$, and by Eq.~(\ref{eqn:Fnweightbound}) we have $w(\varphi)\leq d_X(y,z)+1$, we obtain $$d(\gamma_{y,x}, \gamma_{z,x}) \leq w(\varphi)\leq d_X(y,z)+1.$$
On the other hand, for any $\varphi_1, \ldots, \varphi_n \in F$ with $\varphi_n \cdots \varphi_1(y)=z$, by Eq. (\ref{eqn:Fnweightbound}) and the triangle inequality, we have 
$$\sum_{i=1}^n w(\varphi_i) \geq \sum_{i=1}^n d_X(\varphi_{i-1}\cdots \varphi_1(y), \varphi_i(\varphi_{i-1}\cdots \varphi_1(y))) \geq d_X(y,z),$$
thus 
$d(\gamma_{y,x}, \gamma_{y,x}) \geq d_X(y,z)$. We thus have obtained that
$$ d_X(y,z) \leq d(\Phi(y), \Phi(z)) \leq d_X(y,z)+1,$$
thus the map $\Phi$ is a bijective coarse equivalence with $\rho_-(r)=r$ and $\rho_+(r)=r+1$.
\end{proof}

We end the Subsection with one potential application of the latter construction. 
\begin{remark}\label{remark:BirgetRhodes}
Suppose we have a coarse invariant geometric property (GP), and an algebraic technique to prove that an inverse semigroup has property (GP) then its uniform Roe algebra has property (RP).  We would like to deduce that (GP) implies (RP) for arbitrary metric spaces.

If, for a metric space $X$, we can construct an inverse semigroup $S$ that contains $X$ as a component \textit{and} retains the property (GP), then $C^*_u(S)$ has (RP), and if furthermore (RP) is closed under taking either C*-subalgebras or quotients, then $C^*_u(X)$ (which is both a C*-subalgebra and a quotient) has property (RP) as well. 

Depending on the property (GP), the construction above may not result in an inverse semigroup which retains it. For example, the reader can verify that if $X$ has asymptotic dimension $0$ (see Definition \ref{def:asdim0}), then so does $S$, but if $X$ is sparse (see Definition \ref{def:sparse-extended}), $S$ may not be (in fact it is sparse if and only if $X$ has asymptotic dimension $0$). However, the above construction is just an example of the more general idea of realizing $X$ in an oversemigroup of the rank $1$ maps of $\mathcal I_X$. The interested reader can verify that if we replace the group of bijections in the construction above by its Birget-Rhodes expansion (see \cite{Szen} for the definition) which we equip with the metric inherited naturally from the group, then the resulting metric will inherit both being proper and sparse. We do not give a detailed proof for these arguments as for these particular properties, the respective C*-characterizations are already known for (non-extended) metric spaces by other methods.
\end{remark}

\subsection{Metrics in groupoids} \label{groupoid-metric}
This subsection discusses the relationship between length functions in inverse semigroups, and those in \'etale groupoids, as introduced in~\cite{MW21}. The main result is Theorem~\ref{thm:metric-groupoids-relations} below, which provides a partial dictionary between these two settings.
As this subsection is not referenced in the rest of the paper, and we are expecting that readers interested in these results are already familiar with \'etale groupoids, we have not included a comprehensive introduction, but the reader can find one in~\cite{exel} and references therein. 

We briefly recall the \textit{universal groupoid} associated to an inverse semigroup $S$. This groupoid was introduced in~\cite{P99}, and has been studied extensively since (see, e.g.,~\cite{buss-exel-2009,exel,steinberg-2010} and references therein). Given a (discrete) inverse semigroup $S$, whose semilattice of idempotents we denote by $E$, let $\hat{E}_0$ be the \textit{spectrum} of $E$, that is, the set of non-trivial filters on $E$, equipped with a topology with basis of clopen sets of the form $D_e \cap D_{e_1}^c \cap \cdots \cap D_{e_n}^c$, $e_1, \ldots, e_n < e$, where
$$D_e=\{\xi \in \hat E_0 \colon e \in \xi\}.$$
We see every $e \in E$ as a principal filter $e^\uparrow := \{f \in E \mid f \geq e\} \in \hat{E}_0$. We then have a canonical action $\theta \colon S \curvearrowright \hat{E}_0$ given by 
 partial homeomorphisms $\theta_s \colon D_{s^*s} \rightarrow D_{ss^*}$ where
$$ \theta_s \left(\xi\right) = \left\{e \in E \; \colon \; e \geq sfs^* \; \text{for some} \; f \in \xi\right\} $$
whenever $\xi \in D_{s^*s}$. The \textit{universal groupoid of $S$} is then the groupoid of germs of the action $\theta \colon S \curvearrowright \hat{E}_0$, that is, the set $G$ of equivalence classes $[s, \xi]$ where  $s \in S$ and $\xi \in D_{s^*s} \subset \hat{E}_0$, and $[s, \xi] = [t, \zeta]$ whenever $\xi = \zeta$ and there is some idempotent $e \in E$ such that $\xi \in D_{e}$ and $se = te$. The unit space $G^{(0)}$ of $G$ is then homeomorphic to $\hat{E}_0$, and $G$ is equipped with the topology coming from the unit space such that the resulting $G$ is \'{e}tale. It is then routine to show that $G$ is a locally compact \'{e}tale (in fact, ample) groupoid with Hausdorff unit space. Note, however, that $G$ itself may not be Hausdorff. An efficient characterization of Hausdorffness of $G$ in terms of properties of $S$ is given in~\cite[Theorem~5.17]{steinberg-2010}.

We now introduce the relevant notions about \textit{length functions} in groupoids, as defined in~\cite{MW21}.
\begin{definition} \label{length-func-groupoid}
Let $G$ be a locally compact, Hausdorff \'{e}tale groupoid. We say a function $l \colon G \rightarrow [0, \infty)$ is a \textit{length function} for $G$ if for all $g, h \in G$ the following hold:
\begin{enumerate}
  \item $l(g) = 0$ if and only if $g \in G^{(0)}$;
  \item $l(g) = l(g^{-1})$;
  \item $l(gh) \leq l(g) + l(h)$ whenever $s(g) = r(h)$, i.e., $g$ and $h$ are composable. 
\end{enumerate}
Moreover, we say $l$ is
\begin{enumerate}
\item \textit{proper} if for every $K \subset G \setminus G^{(0)}$, if $\sup_{g \in K} l(g) < \infty$ then $K$ is precompact (i.e. $\ol K$ is compact);
\item \textit{controlled} if for every $K \subset G$, if $K$ is precompact then $\sup_{g \in K} l(g) < \infty$.
\end{enumerate}
\end{definition}

\begin{remark}
In the definition of proper length function above we avoid the unit space in $K$ since we do not assume $G^{(0)}$ to be compact, but only locally compact. Indeed, observe that if $K = G^{(0)}$ is not compact then $K$ cannot be precompact, whereas $\sup_{g \in K} l(g) = 0$. This is the same technicality we have seen in the definition of proper metrics for inverse semigroups (see Definition~\ref{def:metric}), where we only require the finite set $F\Subset S$ to implement paths of non-zero length, or in the case of length functions, where idempotents are excluded from cylinders (see Definition~\ref{def:lengthfunction}).
\end{remark}

\begin{remark}
Ma and Wu~\cite{MW21} study length functions that are in addition continuous, which requires the groupoid to be Hausdorff. Indeed, recall that an \'{e}tale groupoid is Hausdorff if and only if $G^{(0)}$ is closed in $G$. Therefore, if $G$ is not Hausdorff then the preimage $l^{-1}(0)=G^{(0)}$ is not closed. 

Most of what we prove in this section does not use the Hausdorff condition, but following the definition of~\cite{MW21} we will state everything for the Hausdorff case. 
\end{remark}

We record the following result from~\cite[Theorem~4.10]{MW21} for convenience of the reader.
\begin{theorem} \label{thm:mawu}
Up to coarse equivalence, any $\sigma$-compact locally compact Hausdorff \'{e}tale groupoid has a unique proper and controlled continuous length function.
\end{theorem}

\begin{remark} \label{length-loc-const}
By~\cite[Remark~4.11]{MW21}, the length functions on Hausdorff universal groupoids are locally constant. In general, length functions on ample Hausdorff groupoids are locally constant.
\end{remark}

The following definition will turn out to correspond to continuity in the groupoid language.
\begin{definition}
A length function $l \colon S \rightarrow [0, \infty)$ on an inverse semigroup $S$ is \textit{tightly proper} if for every $r \geq 0$ there is some finite $F \Subset S$ witnessing $r$-properness and such that $\max_{s \in F} l(s) \leq r$.
\end{definition}

\begin{proposition} \label{prop:tight-proper}
Any quasi-countable inverse semigroup $S$ has a tightly proper length function.
\end{proposition}
\begin{proof}
It is straightforward to see that the construction of Proposition~\ref{prop:metric-existence} yields a tightly proper length function.
\end{proof}

The following technical results will be useful in the proof of Theorem~\ref{thm:metric-groupoids-relations}.
\begin{lemma} \label{lemma:compact-bisections}
Let $S$ be an inverse semigroup, and let $G$ be its universal groupoid. Then, for any $K \subset G$, the following statements are equivalent:
\begin{enumerate}
\item \label{lemma:compact-bisections:precompact} $K$ is precompact;
\item \label{lemma:compact-bisections:finite} there is a finite $F \Subset S$ and a compact $C \subset G^{(0)}$ such that $K \subset \{[s, x] \mid s \in F, x \in C \cap D_{s^*s}\}$.
\end{enumerate}
\end{lemma}
\begin{proof}
In order to prove that~(\ref{lemma:compact-bisections:precompact}) implies~(\ref{lemma:compact-bisections:finite}), observe that the closure $\overline{K}$ is compact, and covered by the open bisections coming from the inverse semigroup, that is, $\overline{K} \subset \cup_{s \in S} \{[s, x] \mid x \in D_{s^*s}\}$. By compactness there is a finite set $F \subset S$ such that $\overline{K} \subset \cup_{s \in F} \{[s, x] \mid x \in D_{s^*s}\}$. Putting $C := \overline{s^{-1}(K)}$, it follows that $C$ is compact as the sets $D_{s^*s} \subset G^{(0)}$ are compact open, and $C \subseteq \bigcup_{s \in F} D_{s^*s}$ hence $C$ is a closed subset of a compact set. Then the pair $F, C$ satisfies~(\ref{lemma:compact-bisections:finite}).

Suppose now that $F, C$ satisfies~(\ref{lemma:compact-bisections:finite}). Then $\{[s, x] \mid s \in F, x \in C \cap D_{s^*s}\}$ is compact, as it is a finite union of the compact open sets $(s, C \cap D_{s^*s})$, and hence $\overline{K}$ is compact as well, being a closed subset of a compact set.
\end{proof}

\begin{remark}
We observe in passing that Lemma~\ref{lemma:compact-bisections} holds in much greater generality. Indeed, it holds even when $S$ is only a \textit{wide} inverse semigroup of compact open bisections of $G$ (see~\cite[Definition~2.14]{buss-exel-2009}).
\end{remark}

\begin{lemma} \label{lemma:groupoid-convergence}
Let $S$ be an inverse semigroup, and let $G$ be its universal groupoid. Suppose that $\{g_a\}_{a \in A} \subset G$ is a net with $g_a \rightarrow [s, x]$. Then $s(g_a) \rightarrow x \in G^{(0)}$, and $g_a = [s, s(g_a)]$ for all sufficiently large $a$. 
\end{lemma}
\begin{proof}
Put $g_a = [t_a, x_a]$ for some $t_a \in S$ and $x_a \in D_{t_a^*t_a}$. Note that $x_a$ must converge to $x$, for otherwise we would find an open neighborhood $x \in U \subset G^{(0)}$ such that $x_a \not\in U$ for large $a$. In this case, the set $(s, U) := \{[s, y] \mid y \in U\}$ would be an open neighborhood of $[s, x]$, and would contain no point $g_a$, contradicting the assumption that $g_a \rightarrow [s, x]$.

In order to finish the proof, note that $x \in D_{s^*s}$ is an open set 
hence $(s, D_{s^*s})$ is an open neighborhood of $[s,x]$, so $[t_a,x_a] \in (s, D_{s^*s})$ for sufficiently large $a$, that is $g_a=[t_a,x_a] = [s, x_a]$ for sufficiently large $a$ as desired.
\end{proof}

The following theorem states the relationship between length functions (and hence metrics as well) as studied in this paper and those in groupoids as introduced before (coming from~\cite{MW21}). 
\begin{theorem} \label{thm:metric-groupoids-relations}
Let $S$ be a quasi-countable inverse semigroup, and let $G$ be its universal groupoid. Suppose that $G$ is Hausdorff. Then the map
$$ \Phi\left(l\right)\left(\left[s, x\right]\right) = \inf \left\{l(t): t \in S, \; x \in D_{t^*t} \; \text{and} \; \left[t, x\right] = \left[s, x\right]\right\}$$
is an injective map between length functions of $S$ and controlled length functions of $G$,
and 
$$\Psi(l)(s)=\sup_{s \in D_{s^*s}} l([s,x])$$
is its left inverse, i.e. $\Psi \circ \Phi =\id$.

 Moreover, the following assertions hold:
\begin{enumerate}
  \item \label{thm:length-grpd:proper} $l$ is proper if and only if $\Phi(l)$ is proper;
  \item \label{thm:length-grpd:tight} $l$ is tightly proper if and only if $\Phi(l)$ is proper and continuous.
\end{enumerate}
\end{theorem}
\begin{proof}
First, observe that the infimum in the definition of $\Phi(l)$ is actually a minimum, as the image of $l$ is a discrete set, furthermore the minimum is attained on some $l(t)$ with $t \leq s$. The fact that $\Phi(l)$ is a length function follows from the respective properties of $l$.
If $\Phi(l)([s,x]) = 0$ then $[s,x] = [e,x]$ for some $e \in E$, and hence $[s,x] \in G^{(0)}$. Likewise, $\Phi(l)([s,x]) = \Phi(l)([s^*, sx])$ follows from the fact that $l(s) = l(s^*)$ for every $s \in S$. For the triangle inequality, consider a germ $[st,x]=[s,\theta_t(x)][t,x]$, and assume $\Phi(l)([s,\theta_t(x)])=l(u)$ where $[s,\theta_t(x)]=[u,\theta_t(x)]$, and $\Phi(l)([t,x])=l(v)$ where $[t,x]=[v,x]$. Then $[st,x]=[uv,x]$ so 
$$\Phi(l)([st,x]) \leq l(uv) \leq l(u)+l(v)=\Phi(l)([s,\theta_t(x)])+\Phi(l)([t,x])$$
indeed.

Furthermore, observe that $\Phi(l)$ is always controlled (in the sense of Definition~\ref{length-func-groupoid}). In fact, if $K \subset G$ is precompact then by Lemma~\ref{lemma:compact-bisections} we may find a finite $F \Subset S$ and a compact $C \subset G^{(0)}$ such that $\overline{K} \subset \{[s,x] \; \colon \; s \in F \; \text{and} \; x \in C \cap D_{s^*s}\}$. Therefore
$$ \sup_{g \in K} \Phi\left(l\right)\left(g\right) \leq \max_{s \in F} \sup_{x \in C \cap D_{s^*s}} \Phi\left(l\right)\left(\left[s, x\right]\right) \leq \max_{s \in F} l\left(s\right) < \infty, $$
as desired.

Considering the map $\Psi$, note, again, the supremum in $\Psi$ is always a maximum, i.e., it is finite, since $l$ is assumed controlled and the sets $D_{s^*s}$ are compact (and open). 
We show that $\Psi$ indeed maps controlled length functions on $G$ to length functions on $S$ -- the only non-trivial property of $\Psi(l)$ is the triangle inequality. 
As we have $\Psi(l)(st)=l([st,x])$ for some $x \in D_{t^*s^*st}$, then 
$$\Psi(l)(st)\leq l([s,\theta_t(x)])+l([t, x]) \leq \Psi(l)(s)+\Psi(l)(t).$$


We next prove that $\Psi(\Phi(l)) = l$ for every length function $l$ on $S$, which also implies the injectivity of $\Phi$. Indeed, observe we have to prove that
$$ \sup_{x \in D_{s^*s}} \inf_{\substack{t \in S \\ \left[t, x\right] = \left[s, x\right]}} l\left(t\right) = l\left(s\right) $$
for all $s \in S$. First, note that the inequality $\leq$ is immediate, as $\inf_{[t, x] = [s, x]} l(t) \leq l(s)$ for all $x \in D_{s^*s}$. The $\geq$ inequality is more subtle, and follows from the observation that
$$ \sup_{x \in D_{s^*s}} \inf_{\substack{t \in S \\ \left[t, x\right] = \left[s, x\right]}} l\left(t\right) \geq \inf_{\substack{t \in S \\ \left[t, \left(s^*s\right)^\uparrow\right] = \left[s, \left(s^*s\right)^\uparrow\right]}} l\left(t\right) = l\left(s\right), $$
since if $[t, (s^*s)^\uparrow] = [s, (s^*s)^\uparrow]$ then $se = te$ for some $e \in E$ such that $e \geq s^*s$, which implies that $s = ss^*s = se = te$, i.e., $t \geq s$, and so $l(t) \geq l(s)$, finishing the proof.

We now prove assertion~(\ref{thm:length-grpd:proper}). First, suppose that $l$ is a proper length function on $S$, and let $K \subset G \setminus G^{(0)}$ be such that $r := \sup_{g \in K} \Phi(l)(g) < \infty$. That is, for every $g \in K$ there is a $t \in S$ such that $0<l(t) \leq r$, $s(g) \in D_{t^*t}$ and $[t, s(g)] = g$. Now, by the properness of $l$ we may find a finite $F \Subset S$ that upper bounds the cylinder $C_r$, i.e., if $0 < l(t) \leq r$ then $t \leq m$ for some $m \in F$. 
In particular $D_{t^*t} \subseteq D_{m^*m}$ so $s(g)\in D_{m^*m}$ and $[t,s(g)]=[m,s(g)]$.
It then follows that
$$ K \subset \left\{\left[m,s\left(g\right)\right] \; \colon \; s\left(g\right) \in D_{m^*m} \; \text{and} \; m \in F \; \text{and} \; g \in K \right\}, $$
which, by Lemma~\ref{lemma:compact-bisections}, implies that $K$ is precompact, as desired.

For the converse, let $r \geq 0$ be given, and consider the set:
$$ K_r := \left\{\left[s, \left(s^*s\right)^\uparrow\right] \; \colon \; 0 < l\left(s\right) \leq r\right\}. $$
Then $K_r$ contains no units of $G$, and $\sup_{g \in K_r} \Phi(l)(g) \leq r < \infty$. Thus, since $\Phi(l)$ is assumed to be proper, we have that $K_r$ is precompact which, by Lemma~\ref{lemma:compact-bisections}, implies we may find a finite set $F_r \Subset S$ and a compact set $C_r \subset G^{(0)}$ such that $K_r \subset \{[s, x] \mid s \in F_r, x \in C_r \cap D_{s^*s}\}$. We claim that $F_r$ then witnesses $r$-properness of $l$. For this, let $s \in S$ be such that $0 < l(s) \leq r$. Then $[s, (s^*s)^\uparrow] \in K_r$, and hence $[s, (s^*s)^\uparrow] = [m, (s^*s)^\uparrow]$ for some $m \in F_r$. By the germs relations this means that $se = me$ for some $e \in E$ such that $(s^*s)^\uparrow \in D_e$. This implies that $e \in (s^*s)^\uparrow$, meaning that $e \geq s^*s$. In turn, this implies that $s = ss^*s = se = me$, i.e., $F_r \ni m \geq s$, as desired.

We now turn our attention to assertion~(\ref{thm:length-grpd:tight}). Let $l$ be tightly proper and observe that, by~(\ref{thm:length-grpd:proper}), $\Phi(l)$ is a proper length function, so it is enough to prove that it is also continuous. For this, let $\{g_a\}_{a \in A} \subset G$ be a convergent net, say to $[s, x] \in G$. Using the fact that the infimum in the definition of $\Phi(l)$ is in fact a minimum, by choosing the representative of $[s,x]$ on which the minimum is attained, we have $l(s) = \Phi(l)([s, x])$. 
By Lemma~\ref{lemma:groupoid-convergence}, we may, without loss of generality, assume that $g_a = [s, x_a]$, where $x_a = s(g_a)$, and $x_a \rightarrow x$. Then
$$ \Phi\left(l\right)\left(\left[s, x_n\right]\right) = \inf_{\substack{t \in S \\ \left[t, x_n\right] = \left[s, x_n\right]}} l\left(t\right) \leq l\left(s\right) = \Phi\left(l\right)\left(\left[s, x\right]\right). $$
The other inequality follows from the tight properness assumption of $l$. Indeed, suppose that $\Phi(l)([s,x_a]) \not\rightarrow l(s)$, and note that, since the values that $\Phi(l)$ takes are discrete (see~\cite[Lemma~4.14]{MW21}), it must be the case that $\Phi(l)([s,x_a]) \leq r < l(s)$ for some positive numbers $r \geq 0$ in some subnet $B$. Then let $F_r \Subset S$ be a finite set that witnesses the $r$-tight properness of $l$, and observe that $[s, x_a] = [m_a, x_a]$ for some $m_a \in F_r$. Since $F_r$ is finite this implies that, possibly passing to a subnet, we may assume that $[s, x_a] = [m, x_a]$ for all $a$, and hence $[m, x_a] \rightarrow [m, x] = [s, x]$ since $G$ is assumed Hausdorff and limit points are unique in Hausdorff spaces. It would then follow that
$$ l\left(s\right) = \Phi\left(l\right)\left(\left[s, x\right]\right) = \Phi\left(l\right)\left(\left[m, x\right]\right) \leq l\left(m\right) \leq \max_{n \in F_r} l\left(n\right) \leq r < l\left(s\right)$$
by the choice of $F_r$. This is a contradiction, and hence $\Phi(l)([s, x_a]) \rightarrow \Phi(l)[s, x]$, as desired.

For the converse, assume that $l$ is proper, but not tightly proper for some $r$. Let $F$ be the finite set witnessing the properness of $r$. Consider the set
$\ol C_r=\{s \in S: l(s)\leq r\}$. Then there is no finite subset of $\ol C_r$ which is an upper bound of $\ol C_r$, and hence there is a set of elements $s_1, s_2, \ldots$
with $l(s_i) \leq r$ and no upper bound in $\ol C_r$. By properness, there is $m \in F$ such that $s_i < m$ for infinitely many indices $i$ -- let us pass to this subsequence. Then, in particular, $s_i^*s_i \leq m^*m$, and thus $[m, \bigcap_i (s_i^*s_i)^\uparrow] \in G$. 

Notice that $[s_i, (s_i^*s_i)^\uparrow] \to [m, \bigcap_i (s_i^*s_i)^\uparrow]$, and $\Phi([s_i, (s_i^*s_i)^\uparrow])\leq l(s_i) \leq r$. Whereas
$$\Phi([m, \cap_i (s_i^*s_i)^\uparrow])=\inf\{l(t): t^*t \geq s_i^*s_i, m \geq t\}=\inf\{l(t): t \geq s_i\}> r$$
as the chain $s_i$ has no upper bound in $\ol C_r$. Thus $\Phi(l)$ is not continuous.
\end{proof}

\begin{remark}
From Theorem~\ref{thm:metric-groupoids-relations} it is apparent that continuity of a length function on a groupoid is not a coarse invariant, as being tightly proper is not a coarse invariant. However, by Theorem~\ref{thm:mawu}, we can always find a continuous length function on $G$, and by Proposition~\ref{prop:tight-proper} we can always find a tight proper length function on $S$. Furthermore, note both length functions are unique up to coarse equivalence.
\end{remark}

  \section{Uniqueness of the Uniform Roe Algebra of an inverse semigroup} \label{sec:roealg}
     This section is dedicated to the uniform Roe algebra of an inverse semigroup.
     The goal is to show that there are several ways to construct a (usually non-separable) C*-algebra $C_u^*(S)$ from $S$ that inherits much of the geometric aspects of the semigroup. This C*-algebra will then be studied in the upcoming Section~\ref{sec:asymdim-lf}. The first way to construct $C_u^*(S)$ is via the following \textit{canonical action} of $S$ on $\ell^\infty(S)$.
  \begin{definition} \label{def:action}
    Let $S$ be an inverse semigroup and let $s \in S$.
    \begin{enumerate}
      \item Given the idempotent $e \in E$ let $I_{e}$ be the two-sided closed ideal
        $$ I_{e} := \left\{f \in \ell^\infty\left(S\right) \mid \text{supp}\left(f\right) \subset e S\right\}. $$
      \item Given $f \in I_{s^*s}$ let $sf \in I_{ss^*}$ be defined by $(sf)(y) := f(s^*y)$ for any $y \in ss^*S=sS$.
    \end{enumerate}
  \end{definition}
  Note that $s$ defines a $*$-isomorphism $I_{s^*s} \rightarrow I_{ss^*}$, where $f \mapsto sf$, and the corresponding action of $S$ on $\ell^\infty(S)$ is in fact the dual of the Wagner-Preston representation. Moreover, recall that given any action of a discrete inverse semigroup $S$ on a C*-algebra $A$ we may form the \textit{reduced crossed product of $A$ by $S$}, that is, a C*-algebra $A \rtimes_r S$ that encapsulates both $A$ and the action of $S$ on $A$~\cite{buss-exel-2009,buss-martinez-2023}. However, since this construction is quite technical and will not be needed here in full generality, we will only prove that the C*-algebras we shall henceforth consider are indeed reduced crossed products as in~\cite{buss-exel-2009,buss-martinez-2023} in Theorem~\ref{thm:roealg:unique} below. For now, consider the representations given by:
  $$ \pi \colon \ell^\infty(S) \rightarrow \mathcal{B}\left(\ell^2\left(S\right) \otimes \ell^2\left(S\right)\right), \quad \; \pi\left(f\right)\left(\delta_x \otimes \delta_y\right) :=
  \left\{
  \begin{array}{rl}
  f\left(yx\right) \delta_x \otimes \delta_y & \text{if} \;\; xx^* = y^*y, \nonumber \\
  0 & \text{otherwise,}
  \end{array}
  \right.$$
  and 
  $$ 1 \otimes v \colon S \rightarrow \mathcal{B}\left(\ell^2\left(S\right) \otimes \ell^2\left(S\right)\right), \quad \; \left(1 \otimes v\right)\left(s\right)\left(\delta_x \otimes \delta_y\right) :=
  \left\{
  \begin{array}{rl}
  \delta_x \otimes \delta_{sy} & \text{if} \;\; y \in s^*S, \nonumber \\
  0 & \text{otherwise.}
  \end{array}
  \right. $$
Notice that $(1 \otimes v)(s)=1 \otimes v_s$, where $v_s$ is given in Subsection \ref{subsec:pre-invsem}. Furthermore, observe that the representations are covariantly intertwined, that is, they satisfy that:
\[
  \left(1 \otimes v_s\right) {\pi}\left(f\right) \left(1 \otimes v_{s^*}\right) = {\pi} \left(s f\right)
\]
for all $f \in I_{s^*s}$ and $s \in S$. In this setting, let $\mathcal{R}_S$ be the C*-algebra generated by $\{\pi(f) (1 \otimes v_s) \mid s \in S, f \in I_{s^*s}\}$, which is canonically embedded in $\mathcal{B}(\ell^2(S) \otimes \ell^2(S))$.
The importance of the action in Definition~\ref{def:action} comes from the following observation, whose proof is given in~\cite{LM21} (observe that the countability assumption in~\cite{LM21} is never used). Nevertheless, we include a sketch of the proof for convenience.
  \begin{proposition} \label{prop:unifroealg:action}
    Let $S$ be an inverse semigroup, and let $S$ act on $\ell^\infty(S)$ as above. Then
    $$ C^*\left(\ell^\infty\left(S\right) \cdot \left\{v_s \mid s \in S\right\}\right) \cong \mathcal{R}_S, $$
    where the left hand side is sitting in $\mathcal{B}(\ell^2(S))$ and the right hand side in $\mathcal{B}(\ell^2(S) \otimes \ell^2(S))$.
  \end{proposition}
  \begin{proof}
    Consider the operator $u \colon \ell^2(S) \otimes \ell^2(S) \rightarrow \ell^2(S) \otimes \ell^2(S)$ given by:
    $$ u \left(\delta_x \otimes \delta_y\right) = \left\{
                \begin{array}{rl}
                  \delta_x \otimes \delta_{yx} & \text{if} \;\; xx^* = y^*y, \\
                  0 & \text{otherwise}.
                \end{array} \right. $$
    Then a computation shows that $u$ commutes with $1 \otimes v_s$ and $u \pi(f) u^* = (1 \otimes f) uu^*$. The claim of the proposition then follows.
  \end{proof}

  For the following, recall from~\cite{buss-martinez-2023} (and references therein) the construction of the \emph{reduced crossed product} $\ell^\infty(S) \rtimes_r S$, where the action is the one given in Definition~\ref{def:action}. In~\cite[Section~2]{buss-martinez-2023}, however, this construction is carried out for a \emph{Fell bundle} $\mathcal{A} = (A_s)_{s \in S}$ over the inverse semigroup $S$. These Fell bundles form a vastly more general body of objects than the particular action of Definition~\ref{def:action}.
  Since we do not need such generality in this paper, we simply recall that an action of an inverse semigroup $S$ on a C*-algebra $A$ (e.g.\ Definition~\ref{def:action}) yields a Fell bundle $\mathcal{A} = (A_s)_{s \in S}$ as in \cite{buss-exel-2009,buss-martinez-2023} via declaring
  $A_s \delta_s := I_{ss^*} \delta_s$, where $I_{ss^*}$ is as in Definition~\ref{def:action}. The construction detailed in~\cite{buss-exel-2009,buss-martinez-2023} (and references therein) defines a crossed product, which we denote by $\ell^\infty(S) \rtimes_r S$. For the convenience of the reader, we detail this construction below.

  For any $s \in S$ we let $I_{s,1}$ be the (C*-)ideal of $\ell^\infty(S)$ generated by every $I_e$ where $e \leq s$ and $e \in E$ (see \cite[Definition~2.12]{buss-martinez-2023}). Completing in the weak topology yields a (von Neumann) ideal $I_{s,1}^{**} \subset \ell^\infty(S)^{**}$.\footnote{In~\cite{buss-martinez-2023} the double dual $A^{**}$ is denoted as $A''$. However, the notation $A^{**}$ is more common in functional analysis, and hence we stick to it.} We denote the canonical inclusion map of $\ell^\infty(S)$ into $\ell^\infty(S)^{**}$ by $\iota$. Since von Neumann ideals are always unital, we let $1_{s,1}$ be the unit of $I_{s,1}^{**}$. The \emph{conditional expectation} of the action is given in \cite[Definition~2.25]{buss-martinez-2023} as
  \begin{equation} \label{eq:cond-exp}
    P \colon \oplus_{s \in S}^{alg} \ell^\infty\left(ss^*S\right) v_s \to \left(\ell^\infty\left(S\right)\right)^{**}, \;\; (f_s v_s)_{s \in S} \mapsto \sum_{s\in S}\iota (f_s) 1_{s,1},
  \end{equation}
	where $f_s \in \ell^\infty\left(ss^*S\right)$ is viewed naturally as an element of $\ell^\infty(S)$ by extending by $0$ out of $ss^*S$. We must highlight here that the domain of $P$ above is \emph{not} made of $c_0$ sums, but of finite ones. 

  Quotienting the domain of $P$ by its \emph{nucleus}, namely $\mathcal{N}_P = \{x \mid P(x^*x) = 0\}$, yields a *-algebra, which we denote by $\ell^\infty(S) \rtimes_{alg} S$ (see \cite[Proposition~2.27]{buss-martinez-2023}). Then, the C*-algebra $\ell^\infty(S) \rtimes_r S$ is defined to be the unique C*-algebra that contains $\ell^\infty(S) \rtimes_{alg} S$ densely and such that $P$ induces a faithful conditional expectation on $\ell^\infty(S) \rtimes_r S$. This is the definition that we shall use in the following.
  \begin{theorem} \label{thm:roealg:unique}
    Let $S$ be a quasi-countable inverse semigroup, and let $d$ be a proper and right subinvariant metric. Then $\ell^\infty(S) \rtimes_r S$ and $\mathcal{R}_S$ are *-isomorphic. Moreover, if $S$ is a monoid, then they are isomorphic to $C_u^*(S, d)$ as well.
  \end{theorem}
  \begin{proof}
    Note that, by Proposition~\ref{prop:unifroealg:action}, we may see $\mathcal{R}_S$ embedded in $\mathcal{B}(\ell^2(S) \otimes \ell^2(S))$ or $\mathcal{B}(\ell^2(S))$ as needed. 

    We first prove that $\mathcal{R}_S$ is contained in $C_u^*(S, d)$ as subsets of $\mathcal{B}(\ell^2(S))$. 
    In fact, we prove that $\operatorname{span}\left(\ell^\infty\left(S\right) \cdot \left\{v_s \mid s \in S\right\}\right)$ consists exactly of finite propagation operators in $\mathcal{B}(\ell^2(S))$, which by taking the C*-closures of both sides, implies the statement.
    It is clear that every generator $f v_s$ of $\mathcal{R}_S$ is of finite propagation. Indeed, note that if $\langle f v_s \delta_x, \delta_y \rangle \neq 0$ then $sx = y$ and $s^*sx = x$, which implies that $x \lrel y$ and $d(x, y) \leq d(s^*s, s)$ by right subinvariance of $d$. It follows that every operator in $\operatorname{span}\left(\ell^\infty\left(S\right) \cdot \left\{v_s \mid s \in S\right\}\right)$ is of finite propagation.

    We now prove the reverse inclusion when $S$ is a monoid, which is more subtle. Let $T \in C_u^*(S, d)$ be of propagation $r > 0$. As $d$ is proper and $S$ is a monoid there is some finite $F \Subset S$ such that for any $x,y \in S$ with  $d(x, y) \leq r$, we have $y = tx$ for some $t \in F$. Let $\{t_{x,y}\}_{x, y \in S} \subset F$ be such a choice, i.e., $t_{x,y} x = y$. Then for each $s \in F$,
    we may define an operator $f_s \in \ell^\infty (S)$ by
     $$f_s\left(y\right) :=
          \begin{cases}
            \langle T \delta_{s^*y}, \delta_y \rangle & \text{if} \;\; y \lrel s^*y \;\; \text{and} \;\; s = t_{s^*y,y}, \\
            0 & \text{otherwise,}
          \end{cases}
        , $$
     and it is straightforward to check that $T = \sum_{s \in F} f_sv_s$ (a proof is also given in~\cite[Proof of Theorem 3.25]{LM21}). 

   We now turn to the proof that $\mathcal{R}_S \cong \ell^\infty(S) \rtimes_r S$, regardless of whether $S$ is a monoid. In order to prove this, we shall show that $\mathcal{R}_S$ satisfies the defining feature of $\ell^\infty(S) \rtimes_r S$, namely it is the (necessarily unique) C*-algebra densely containing $\ell^\infty(S) \rtimes_{alg} S$ and such that the conditional expectation $P$ given in Eq.~\ref{eq:cond-exp} induces a faithful conditional expectation on $\mathcal{R}_S$.
   
   We first begin by showing that the image of $P$ is contained in $\ell^\infty(S)$. Indeed, given $s \in S$ let $p_s$ be the orthogonal projection onto $\ell^2(\cup_{e \in s^{\downarrow} \cap E} eS)$, which is a closed subspace of $\ell^2(S)$, and $p_s \in \ell^\infty(S)$. It hence suffices to prove that $1_{s,1} = \iota(p_s)$, where $\iota \colon \ell^\infty(S) \hookrightarrow \ell^\infty(S)^{**}$ is the canonical inclusion. But this is readily done. Indeed, on one hand, $1_{s,1} \leq \iota(p_s)$, since $p_s x = x$ and $x p_s = x$ for all $x \in I_e$ and idempotent $e \leq s$. Likewise, note that $1_{s,1} \geq \iota(p_s)$. For this, fix some idempotent $e \leq s$ and observe that $1_{s,1} \geq \iota(p_e) = \iota(v_e)$, since $I_e \subseteq I_{s,1} \subseteq I_{s,1}^{**}$. Thus $\iota(p_s) = \sup_{e \leq s^\downarrow \cap E} \iota(p_e) \leq 1_{s,1}$, as desired. This means we may view $P$ in Eq.~(\ref{eq:cond-exp}) simply as the map
   $$P \colon \oplus_{s \in S}^{alg} \ell^\infty\left(ss^*S\right) v_s \to \ell^\infty\left(S\right), \;\; (f_s v_s)_{s \in S} \mapsto \sum_{s\in S} f_s p_s.$$
   Next, recall that the usual conditional expectation $E \colon \mathcal{B}(\ell^2(S)) \to \ell^\infty(S)$ is defined as $E(t)(x) = \langle t\delta_x, \delta_x\rangle$ for all $x \in S$. Moreover, it is not hard to show that $E$ is \textit{faithful}, meaning that $t = 0$ whenever $E(t^*t) = 0$.
   In this context, consider the map
   $$\psi \colon \oplus_{s \in S}^{alg} \ell^\infty\left(ss^*S\right) v_s \to \mathcal{B}(\ell^2(S)), \;\; (f_sv_s)_{s \in S} \mapsto \sum_{s \in S} f_s v_s.$$
   We claim that $P=E \circ \psi$. By linearity, it suffices to prove that $P(f_s v_s) = E(\psi(f_s v_s))$ for all $s \in S$ and $f_s \in \ell^\infty(ss^*S)$. In order to do this, note that $\psi(f_s v_s) = f_s v_s$ (as an element of $\mathcal{B}(\ell^2(S))$, not as a formal monomial). Then, for all $x \in S$ we have that
   $$ P\left(f_s v_s\right)\left(x\right) = \left(f_s p_s\right)\left(x\right) =
   \left\{
    \begin{array}{rl}
      f_s\left(x\right) & \text{if} \; x \in \cup_{e \in s^\downarrow \cap E} eS, \\
      0 & \text{otherwise.}
    \end{array}
   \right.
  $$
  Likewise,
  $$
    E\left(\psi\left(f_s v_s\right)\right)\left(x\right) = E\left(f_s v_s\right)\left(x\right) = \langle f_s v_s \delta_x, \delta_x \rangle =
     \left\{
      \begin{array}{rl}
        f_s\left(x\right) & \text{if} \; x = s^*s x = sx, \\
        0 & \text{otherwise.}
      \end{array}
     \right.
  $$
  By simply noting that $x \in \cup_{e \in s^\downarrow \cap E} eS$ if and only if $x = s^*s x = sx$, the proof that $P = E \circ \psi$ is done.
   Notice that the image of $\psi$ consists exactly of the set $\operatorname{span}\left(\ell^\infty\left(S\right) \cdot \left\{v_s \mid s \in S\right\}\right)$ of finite propagation operators, since $\bigcup_{s \in S} ss^\ast S=S$. All that remains to be shown is $\mathcal N_P=\ker\psi$.  This would indeed imply that $\operatorname{Im}(\psi)=\ell^\infty(S) \rtimes_{alg} S \cong  \oplus_{s \in S}^{alg} \ell^\infty\left(ss^*S\right)/{\mathcal N_P}$, which $\mathcal R_S$ densely contains, and the conditional expectation induced by $P$ on $\mathcal R_S$ will coincide with $E$, which is faithful on all of $\mathcal{B}(\ell^2(S))$.
   
   In order to prove that $\ker \psi = \mathcal{N}_P$ we first must show that $\psi$ is a homomorphism. For this, given $f_s v_s$ and $f_t v_t$ we let
   $$
    f_s v_s f_t v_t := \left[x \mapsto f_s\left(stx\right) f_t\left(tx\right)\right] v_{st} \;\; \text{and} \;\; \left(f_sv_s\right)^* := \left[x \mapsto \overline{f_s\left(s^*x\right)} \right] v_{s^*}.
   $$
   The above operations, extended by linearity, equip $\oplus_{s \in S}^{alg} \ell^\infty(ss^*S) v_s$ with an $*$-algebra structure. Moreover, note that $\psi$ is then a $*$-homomorphism. Given any $x$ such that $\psi(x) = 0$, it follows from $P = E \circ \psi$, that
   $$
    P\left(x^*x\right) = E\left(\psi\left(x^*x\right)\right) = E\left(\psi\left(x^*\right) \psi\left(x\right)\right) = 0,
   $$
   and hence $\ker \psi \subseteq \mathcal{N}_P$. For the reverse inclusion, let $x$ be such that $P(x^*x) = 0$. Then, by the computation above and the fact that $E$ is faithful, we have that $\psi(x) = 0$, as desired.
\end{proof}

In view of Theorems~\ref{thm:metric} and~\ref{thm:roealg:unique}, we shall henceforth fix a proper and right subinvariant metric $d$ on $S$, and call $C_u^*(S, d)$ the \textit{uniform Roe algebra} of $S$, as it (essentially) does not depend on the choice of such a metric.

\begin{remark}
  In order to answer a question of the referee, the isomorphisms given in Theorem~\ref{thm:roealg:unique} do fix the copies of the natural abelian subalgebra $\ell^\infty(S)$ in all of $\ell^\infty(S) \rtimes_r S, \mathcal{R}_S$ and $C_u^*(S, d)$ (at least when $S$ is a monoid). Indeed, for instance, $\psi$ in the proof of Theorem~\ref{thm:roealg:unique} does send $\ell^\infty(S)$ into $\mathcal{B}(\ell^2(S))$ in the usual way. 
\end{remark}

  \section{Local finiteness properties, geometry and quasi-diagonality} \label{sec:asymdim-lf}
  This final section of the text has the goal of characterizing those inverse semigroups that have asymptotic dimension $0$, both algebraically and by properties of their uniform Roe algebra. Incidentally, the same methods actually allow us to characterize those inverse semigroups whose $\lrel$-classes are \textit{sparse}. 

  \subsection{Local finiteness, asymptotic dimension \texorpdfstring{$0$}{0} and strong quasi-diagonality} \label{subsec:asymdim-lf}
  Recall we say an inverse semigroup $S$ is \textit{locally finite} if every finitely generated inverse subsemigroup of $S$ is finite. Observe that when $S$ happens to have an identity, it is locally finite as an inverse monoid if and only if it is locally finite as an inverse semigroup, that is, its finitely generated inverse submonoids are finite if and only if its finitely generated inverse subsemigroups are finite. We therefore do not differentiate between these two notions.

  Asymptotic dimension was introduced by Gromov as an analogue of topological dimension for metric spaces. It is a coarse invariant of the space. As for the purposes of this paper we are only interested in asymptotic dimension $0$, this is the only definition that we shall give (see~\cite{NY12} for the general definition).
  \begin{definition}\label{def:asdim0}
    Let $(X, d)$ be a metric space of bounded geometry.
    \begin{enumerate}
      \item A \textit{cover} $\mathcal{U}$ of $X$ is a family of subsets $U \subset X$ such that $\cup_{U \in \mathcal{U}} U = X$.
      \item The cover $\mathcal{U}$ is \textit{uniformly bounded} if the diameters of the sets $U \in \mathcal{U}$ are uniformly bounded, that is, $\sup_{U \in \mathcal{U}} \diam(U) < \infty$.
      \item We say $(X, d)$ is of \textit{asymptotic dimension $0$} if for all $r \in \mathbb R^+$ the space $X$ has a uniformly bounded cover $\mathcal{U}$ such that $d(U,V)>r$ for any $U, V \in \mathcal{U}$ with $U \neq V$.
    \end{enumerate}
  \end{definition}
  It is well known that a group is locally finite if and only if it has asymptotic dimension $0$ (see~\cite[Theorem~2]{Smith}). We shall prove the analogous result for inverse semigroups.

    Below we give a known alternative definition for having asymptotic dimension $0$. Recall that, given a metric space $(X,d)$, we may define an equivalence $\sim_r$ on $X$ by setting $x \sim_r y$ if there exists a sequence of points $z_1,\ldots, z_n$ with $x = z_1$, $y = z_n$ and $d(z_i, z_{i+1}) \leq r$. The $\sim_r$-classes are called \emph{$r$-components}. 
  \begin{lemma} \label{lemma:asdim0r}
  	The following are equivalent for any metric space $(X,d)$ with bounded geometry:
  	\begin{enumerate}
  		\item \label{item:asdim0r:asdim} $X$ has asymptotic dimension $0$;
  		\item \label{item:asdim0r:bounded} for any $r \in \mathbb R^+$, the $r$-components of $X$ have uniformly bounded diameter;
  		\item \label{item:asdim0r:sizebounded} for any $r \in \mathbb R^+$, the $r$-components of $X$ have uniformly bounded size.
  	\end{enumerate}
  \end{lemma}
  \begin{proof}
  	We will prove that~(\ref{item:asdim0r:asdim})~$\Rightarrow$~(\ref{item:asdim0r:sizebounded})~$\Rightarrow$~(\ref{item:asdim0r:bounded})~$\Rightarrow$~(\ref{item:asdim0r:asdim}). For~(\ref{item:asdim0r:asdim})~$\Rightarrow$~(\ref{item:asdim0r:sizebounded}), fix $r \in \mathbb R^+$ and let $\mathcal{U}$ be a covering witnessing the asymptotic dimension of $X$, that is, there is some $k \geq 0$ such that $\diam(U) \leq k$ and $d(U,V)>r$ for any $U \neq V \in \mathcal U$. In particular if $x \in U$ then for any $z \in U$ with $d(x,z) \leq r$ we have $z \in U$ as well. Hence, by an inductive argument, the $r$-component of $x$ is thus contained in $U \subseteq B_k(x)$, and so its cardinality is bounded by $|B_k(x)|$. Bounded geometry then gives a uniform bound on the size of the $r$-components.
  	
  	For~(\ref{item:asdim0r:sizebounded})~$\Rightarrow$~(\ref{item:asdim0r:bounded}), assume any $r$-component $C$ has size at most $N$. Let $x, y \in C$, and consider a sequence $z_1, \ldots, z_n$ with $x=z_1$, $y=z_n$ and $d(z_i, z_{i+1}) \leq r$. We can without loss of generality assume that the points in the sequence are pairwise distinct, and as they are all contained in $C$, we have $n \leq N$. Then $d(x,y) \leq rN$ by the triangle inequality, so $\diam(C) \leq rN$.
  	
  	Finally, for~(\ref{item:asdim0r:bounded})~$\Rightarrow$~(\ref{item:asdim0r:asdim}), simply note that the set of $r$-components is a cover with the required properties.
  \end{proof}

  The statements of the following lemma are immediate from the definition and the fact that asymptotic dimension is a coarse invariant, along with the fact that proper and right subinvariant metrics in inverse semigroups are unique up to bijective coarse equivalence (see Theorem~\ref{thm:metric}).
  \begin{lemma} \label{lemma:asymdim-zero:unital}
    Let $S$ be an inverse semigroup, and let $S^1$ be the monoid obtained by adjoining an identity to $S$. Let $d, d^1$ be proper and right subinvariant metrics on $S$ and $S^1$ respectively. Then the following hold:
    \begin{enumerate}
      \item $S$ is locally finite if and only if $S^1$ is locally finite.
      \item $(S, d)$ has asymptotic dimension $0$ if and only if $(S^1, d^1)$ has asymptotic dimension $0$.
    \end{enumerate}
  \end{lemma}

  A related, more general concept is that of \textit{sparse} metric spaces (see, e.g.,~\cite{BFV20}). A non-extended metric space $(X,d)$ of bounded geometry is \emph{sparse} if $X$ is the disjoint union of nonempty finite subsets $\{X_n\}_{n \in \mathbb{N}}$ such that $d(X_n, X_m) \to \infty$ as $n+m \to \infty$.
In particular sparse non-extended metric spaces are countable. As we are interested in quasi-countable inverse semigroups, we consider the following definition for (extended) metric spaces.  
  
  \begin{definition} \label{def:sparse-extended}
    Let $(X, d)$ be a metric space of bounded geometry. We say $X$ is \textit{sparse} if $(C, d|_C)$ is a sparse metric space for every connected component $C \subset X$.
  \end{definition}
   The following lemma, which is analogous to Lemma~\ref{lemma:asdim0r}, details the relationship between asymptotic dimension $0$ and sparseness.
  \begin{lemma} \label{lemma:sparser}
    The following are equivalent for any extended metric space $(X,d)$ with bounded geometry:
    \begin{enumerate}
      \item \label{item:sparser:sparse} $X$ is sparse;
      \item \label{item:sparser:diameter} for any $r \in \mathbb R^+$, the $r$-components of $X$ have finite diameter;
      \item \label{item:sparser:bounded} for any $r \in \mathbb R^+$, the $r$-components of $X$ are finite.
    \end{enumerate}
    In particular, any metric space with asymptotic dimension $0$ is sparse.
  \end{lemma}
  \begin{proof}
    The fact that~(\ref{item:sparser:diameter}) and~(\ref{item:sparser:bounded}) are equivalent follows routinely from the fact that $X$ is of bounded geometry. In order to see~(\ref{item:sparser:bounded})~$\Rightarrow$~(\ref{item:sparser:sparse}), let $C\subset X$ be a connected component of $X$, and fix a point $p \in C$. We shall define pairwise disjoint sets $X_i$ such that
    $$C=\bigcup\limits_{n=1}^\infty X_n \;\; \text{and} \;\; d(X_n, X_m) > \max\{n,m\} - 1$$
    when $n \neq m$. Let $X_1$ be the $1$-component of $p$. For $n \geq 2$, define $X_n$ as the $n$-component of $p$ minus the $(n-1)$-component of $p$. By (\ref{item:sparser:bounded}), these are all finite sets. It is also clear from the definition that $C$ is the disjoint union of the sets $\{X_n\}_{n \in \mathbb N}$. Now if $n > m$ then 
    $$d(X_n, X_m) \geq  d(X_n, X_1 \cup \ldots \cup X_{n-1}) \geq d(X \setminus \{X_1 \cup \ldots \cup X_{n-1}\}, X_1 \cup \ldots \cup X_{n-1})>n-1,$$ because
    $X_1 \cup \ldots \cup X_{n-1}$ is exactly the $(n-1)$-component of $p$. Thus $d(X_n, X_m) \geq \max\{n,m\} - 1$ whenever $n \neq m$.

    In order to prove that~(\ref{item:sparser:sparse}) implies~(\ref{item:sparser:bounded}), given $r \in \mathbb R^+$, each $r$-component of $X$ is contained in a disjoint union of nonempty finite subsets $\{X_n\}_{n\in\mathbb{N}}$ such that $d(X_n,X_m)\rightarrow\infty$ as $n+m\rightarrow\infty$. Let $k \in \mathbb N$ be large enough so that if $n+m \geq k$, then $d(X_n,X_m) > r$. We claim that then each $r$-component is contained in one of the finite sets
    $$X_1 \cup \cdots \cup X_{k-1}, \; X_k, \; X_{k+1}, \; \ldots.$$
    Indeed, by the assumption, the distance between any pair of the above sets is bigger than $r$, meaning no two $r$-related elements can fall into two distinct sets.
  \end{proof}

  It is immediate from the above lemma that sparseness, too, is a coarse invariant property. Note that even though bounded geometry metric spaces with asymptotic dimension $0$ are all sparse, the converse is in general false. However, within the realm of groups, the two notions coincide, as is well known and we will see below. We will show that for inverse semigroups having asymptotic dimension $0$ and being sparse are not equivalent.

   It turns out that the corresponding algebraic property of inverse semigroups is what we term \textit{local $\lrel$-finiteness}. We say an inverse semigroup $S$ is \textit{$\mathcal L$-finite} if all its $\mathcal L$-classes are finite and, as usual, we say it is \textit{locally $\lrel$-finite} if all of its finitely generated inverse subsemigroups are $\lrel$-finite. Again, in the case when $S$ is a monoid, this is equivalent to the finitely generated submonoids of $S$ being $\lrel$-finite. The following characterization of $\lrel$-finite inverse semigroups is useful to know.
  \begin{lemma} \label{lemma:l-finite} The following are equivalent for any inverse semigroup $S$:
    \begin{enumerate}
    \item \label{lemma:l-finite:l} $S$ is $\mathcal L$-finite.
    \item \label{lemma:l-finite:r} $S$ only has finite $\mathcal R$-classes.
    \item \label{lemma:l-finite:d} $S$ only has finite $\mathcal D$-classes.
    \end{enumerate}
  \end{lemma}
  \begin{proof}
    The equivalence of~(\ref{lemma:l-finite:l}) and~(\ref{lemma:l-finite:r}) follows from the fact that $s \lrel t$ if and only if $s^* \rrel t^*$, so the $\lrel$-class $L_s$ is in bijection with the $\rrel$-class $R_{s^*}$.

    Since $\drel$-classes are unions of $\lrel$-classes, the~(\ref{lemma:l-finite:d}) $\Rightarrow$~(\ref{lemma:l-finite:l}) implication is trivial. For the converse, let $s \in S$ and note that $D_s=\cup_{t \in R_s} L_t = \cup_{t^* \in L_{s^*}}L_t$, whence $|D_s|\leq |L_s|^2$.  
  \end{proof}

  Observe that for groups, the notions of (local) finiteness and (local) $\lrel$-finiteness coincide, as $\lrel$ is the universal relation. However, for inverse semigroups these are different classes. For instance, finitely generated free inverse semigroups are $\mathcal L$-finite (but not finite), and free inverse semigroups are locally $\lrel$-finite (but not locally finite). The analogue of Lemma~\ref{lemma:asymdim-zero:unital} remains true in this setting, the proof is again immediate from the definitions.
  \begin{lemma} \label{lemma:box:unital}
    Let $S$ be an inverse semigroup, and let $S^1$ be the monoid obtained by adjoining an identity to $S$. Let $d, d^1$ be proper and right subinvariant metrics on $S$ and $S^1$ respectively. Then the following hold:
    \begin{enumerate}
      \item $S$ is locally $\lrel$-finite if and only if $S^1$ is $\lrel$-locally finite.
      \item $(S, d)$ is sparse if and only if $(S^1, d^1)$ is sparse.
    \end{enumerate}
  \end{lemma}
  
  The following lemma characterizes those finitely generated inverse semigroups which are sparse spaces and, respectively, of asymptotic dimension $0$.
  \begin{lemma} \label{lemma:finite-and-asymdim-zero}
    Let $T$ be a finitely generated inverse semigroup.
    \begin{enumerate}
      \item \label{lemma:finite-and-asymdim-zero:finit} If $T$ has asymptotic dimension $0$, then it is finite.
      \item \label{lemma:finite-and-asymdim-zero:l-fin} If $T$ is sparse then it is $\mathcal L$-finite.
    \end{enumerate}
  \end{lemma}
  \begin{proof}
    Since $T$ is finitely generated we obtain the unique right subinvariant proper metric $d$ by taking the word metric with respect to a finite generating set $A \Subset T$. Notice that in this metric, $d(x,y) < \infty$ if and only if they are in the same $1$-component, that is, the $1$-components are the $\mathcal L$-classes. This immediately implies~(\ref{lemma:finite-and-asymdim-zero:l-fin}) by Lemma \ref{lemma:sparser}. Likewise, by Lemma~\ref{lemma:asdim0r}, if $T$ has asymptotic dimension $0$ then its $\mathcal L$-classes are uniformly bounded in size by some constant $k$. Hence, by Lemma~\ref{lemma:l-finite}, in order to prove~(\ref{lemma:finite-and-asymdim-zero:finit}) all we have to do is show that $T$ admits only finitely many $\drel$-classes.

      Consider a Sch\"{u}tzenberger graph $\sch s$ for $s \in T$. It has at most $k$ vertices, and so given any label $a \in A$, there are at most $k$ edges labeled by $a$, since no two edges with the same label can have the same initial (or terminal) vertex. Since $A$ is finite we see that there are only finitely many such edge-labeled graphs up to isomorphism respecting the labels. Recall that $\sch{s}$ and $\sch{t}$ are isomorphic as edge-labeled digraphs if and only if $s \drel t$, which implies $T$ has only finitely many $\drel$-classes.
   \end{proof}

  The following lemma is routine to show, but will be useful when proving the main contributions of the section (see Theorems~\ref{thm:asymdim-zero} and~\ref{thm:boxspace}).
  \begin{lemma} \label{lemma:contractive-embedding}
    Given any uniformly bounded metric spaces $(X, d_X)$ and $(Y,d_Y)$ with $X \subseteq Y$, if for any  $x, z \in X$ we have $d_X(x,z) \geq d_Y(x,z)$, then 
    \begin{enumerate}
      \item \label{lemma:contractive-embedding:dim0} if $Y$ has asymptotic dimension $0$, then so does $X$;
      \item \label{lemma:contractive-embedding:sparse} if $Y$ is sparse, then so is $X$.
    \end{enumerate}
  \end{lemma}
  \begin{proof}
    Suppose $x$ and $z$ are in the same $r$-component of $X$, that is, there exist $x_1, \ldots, x_n \in X$ such that $x=x_1, z=x_n$, and $d_X(x_i,x_{i+1})\leq r$. Then $d_Y(x_i,x_{i+1})\leq r$, so $x$ and $z$ are also in the same $r$-component in $Y$. This shows that the $r$-components of $X$ are contained in the $r$-components of $Y$, which implies both statements~(\ref{lemma:contractive-embedding:dim0}) and~(\ref{lemma:contractive-embedding:sparse}) by Lemmas~\ref{lemma:asdim0r} and~\ref{lemma:sparser} respectively.
  \end{proof}

  We obtain the following key lemma as an easy consequence.
  \begin{lemma} \label{lem:subsemigroup_inherits}
    Given any inverse semigroup $S$ and an inverse subsemigroup $T \subset S$, 
    \begin{enumerate}
      \item \label{lem:subsemigroup_inherits:asdim0} if $S$ has asymptotic dimension $0$, then so does $T$;
      \item \label{lem:subsemigroup_inherits:sparse} if $S$ is sparse, then so is $T$.
    \end{enumerate}
  \end{lemma}
  \begin{proof}
  Denote by $d_s$ and $l_s$ the metric and the respective length function on $S$.
    Choose a generating set $A$ for $T$, and consider the metric on $T$ arising as the weighted word metric choosing weights $w(a)=l_s(a)$ for $a \in A$ (see Proposition~\ref{prop:weigthed-metric}). Denote this metric by $d_T$. We claim that if $t_1, t_2 \in T$, then $d_T(t_1, t_2) \geq d_S(t_1,t_2)$. First notice that $t_1$ and $t_2$ are $\lrel$-related in both $S$ and $T$ if and only if $t_1^*t_1 = t_2^*t_2$, so the connected components are the same in both metrics. Assume $t_1 \lrel t_2$, and consider a geodesic path with respect to $d_T$ from $t_1$ to $t_2$ labeled by $a_n, \ldots, a_1$ with $a_i \in A \cup A^*$, that is, $t_2 = a_1 \ldots a_n t_1$ with 
$$d_T (t_1, t_2) = w(a_1) + \ldots + w(a_n)=l_S(a_1) + \ldots + l_S(a_n)\geq l_S(a_1 \ldots a_n) \geq d_S(t_1,t_n)$$ by Lemma \ref{lemma:metric-right-inv} (\ref{lemma:metric-right-inv:contract}). We can then apply Lemma~\ref{lemma:contractive-embedding} to $(T,d_T)$ and $(S,d_S)$, yielding both statements. 
  \end{proof}

  The following lemma provides a useful necessary condition for inverse semigroups not of asymptotic dimension $0$. In particular, it says that such semigroups must contain arbitrarily long $r_0$-paths, for some $r_0 \geq 0$. This was proved in~\cite[Lemma~2.4]{LW18} for non-extended metric spaces, and can be proved for extended metric spaces similarly.
  \begin{lemma} \label{lemma:asymdim-one}
    Let $S$ be an inverse semigroup not of asymptotic dimension $0$. Then there is some $r_0 \geq 0$ and sets $X_n = \{x_1^{(n)}, \dots, x_{m(n)}^{(n)}\}$ such that $m(n) \rightarrow \infty$ when $n \rightarrow \infty$ and the following assertions hold:
    \begin{enumerate}
      \item $d(x_i^{(n)}, x_{i+1}^{(n)}) \leq 2r_0$ and $d(x_1^{(n)}, x_i^{(n)}) \in [(i-1)r_0, ir_0)$ for every $n \in \mathbb{N}$ and $i = 1, \dots, m(n) - 1$.
      \item The sequence $\{\inf_{m \neq n} d(X_n, X_m)\}_{n \in \mathbb{N}}$ is positive and tends to $\infty$ when $n \rightarrow \infty$.  
    \end{enumerate}
    In particular, there are $\lrel$-classes $L_n \subset S$ such that $X_n \subset L_n$.
  \end{lemma}
  
  We are now in a position to give the proof of the main contribution of the subsection. The following theorem characterizes the geometric property of having asymptotic dimension $0$ via the algebraic property of being locally finite, and via the C*-property of having a strongly quasi-diagonal uniform Roe algebra.
  \begin{theorem} \label{thm:asymdim-zero}
    Let $S$ be a quasi-countable inverse semigroup equipped with its unique proper and right subinvariant metric $d$. The following statements are equivalent:
    \begin{enumerate}
      \item \label{thm:asymdim-zero:lf} $S$ is locally finite.
      \item \label{thm:asymdim-zero:zero} $(S, d)$ has asymptotic dimension $0$.
      \item \label{thm:asymdim-zero:limit} $\unifroealg$ is an inductive limit of finite dimensional C*-sub-algebras.
      \item \label{thm:asymdim-zero:locaf} $\unifroealg$ is local AF.
      \item \label{thm:asymdim-zero:strqd} $\unifroealg$ is strongly quasi-diagonal.
    \end{enumerate}
  \end{theorem}
  \begin{proof}
    We shall prove that~(\ref{thm:asymdim-zero:lf})~$\Leftrightarrow$~(\ref{thm:asymdim-zero:zero}), 
    and~(\ref{thm:asymdim-zero:zero})~$\Rightarrow$~(\ref{thm:asymdim-zero:limit})~$\Rightarrow$~(\ref{thm:asymdim-zero:locaf})~$\Rightarrow$~(\ref{thm:asymdim-zero:strqd})~$\Rightarrow$~(\ref{thm:asymdim-zero:zero}).

    Observe that, by Lemmas~\ref{lemma:unital-technicality} and~\ref{lemma:asymdim-zero:unital}, in order to prove~(\ref{thm:asymdim-zero:lf})~$\Leftrightarrow$~(\ref{thm:asymdim-zero:zero}) we may assume that $S$ is a monoid. In this setting, the implication~(\ref{thm:asymdim-zero:zero})~$\Rightarrow$~(\ref{thm:asymdim-zero:lf}) is an immediate consequence of Lemmas~\ref{lemma:finite-and-asymdim-zero}~(\ref{lemma:finite-and-asymdim-zero:finit}) and~\ref{lem:subsemigroup_inherits}~(\ref{lem:subsemigroup_inherits:asdim0}).

    For~(\ref{thm:asymdim-zero:lf})~$\Rightarrow$~(\ref{thm:asymdim-zero:zero}), given $r > 0$ we will show that the $r$-components of $S$ are uniformly bounded in size which, by Lemma~\ref{lemma:asdim0r}, proves~(\ref{thm:asymdim-zero:zero}). Since $d$ is proper there is some finite $F \Subset S$ such that $y \in Fx$ for every $x, y \in S$ such that $d(x, y) \leq r$. Indeed, just put $F := \{1\} \cup \tilde{F}$ where $\tilde{F}$ is as in Definition~\ref{def:metric}~(\ref{def:metric:pr}). The inverse semigroup $M := \langle F\rangle \subset S$ is then finite, as it is finitely generated. Suppose $s$ and $t$ are in the same $r$-component of $S$, i.e., there are elements $x_1, \ldots, x_n \in S$ such that $x_1 = s$, $x_n = t$, and $d(x_i, x_{i+1}) \leq r$. Then $x_{i} \in Fx_{i+1} \subseteq Mx_{i+1}$ by assumption. Therefore, by induction, $s \in Mt$, since $M$ is closed under multiplication. We hence have that the $r$-component of $t$ is contained in $Mt$ and thus has size at most $|M|$, proving the claim.

    The proof of~(\ref{thm:asymdim-zero:lf})~$\Rightarrow$~(\ref{thm:asymdim-zero:limit}) is the exact same as that of~\cite[Theorem~2.2]{LW18} (1) $\Rightarrow$ (2), which is the analogous statement for non-extended metric spaces.

    Lastly, since the implications~(\ref{thm:asymdim-zero:limit})~$\Rightarrow$~(\ref{thm:asymdim-zero:locaf})~$\Rightarrow$~(\ref{thm:asymdim-zero:strqd}) are well known to hold for general C*-algebras, let us prove that~(\ref{thm:asymdim-zero:strqd})~$\Rightarrow$~(\ref{thm:asymdim-zero:zero}). We shall prove the contrapositive, and assuming $\asdim(S) >0$ we will construct a homomorphism $\pi \colon \unifroealg \rightarrow B$ where $\pi(\unifroealg)$ contains a proper isometry. This implies there is a non-quasi-diagonal quotient of $\unifroealg$ (see~\cite[Proposition~7.1.15]{BO08}). First, since $S$ is not of asymptotic dimension $0$, let $r_0 \geq 0, \{X_n\}_{n \in \mathbb{N}}$ and $\{L_n\}_{n \in \mathbb{N}}$ be as in Lemma~\ref{lemma:asymdim-one}. Consider
    $$\textstyle \psi \colon \unifroealg \rightarrow C_u^*\left(\bigsqcup_n L_n, d\right), \;\; x \mapsto xp, $$
    where $p \in \mathcal{B}(\ell^2(S))$ denotes the orthogonal projection onto $\ell^2(\bigsqcup_n L_n)$. In particular, note that $x p = px$ for all $x \in \unifroealg$, and, therefore, $\psi$ is a surjective homomorphism. Now, let $(Y, \tilde{d})$ be any non-extended metric space such that $Y = \bigsqcup_{n \in \mathbb{N}} L_n$ and $\tilde{d}$ restricts to $d$ in $L_n$, whereas $\tilde{d}(L_n, L_m) \rightarrow \infty$ when $n+m \rightarrow \infty$ and $n \neq m$. We will compose $\psi$ with the natural inclusion map
    $$\textstyle \iota \colon C_u^*\left(\bigsqcup_n L_n, d\right) \hookrightarrow C_u^*(Y, \tilde{d}). $$
    The map $\iota \circ \psi$ is then a homomorphism, and its image contains all the operators $t \in C_u^*(Y, \tilde{d})$ of finite propagation and that respect the sets $L_n \subset Y$, that is, $tq_n = q_nt$, where $q_n \in C_u^*(Y, \tilde{d})$ projects onto $\ell^2(L_n) \subset \ell^2(Y)$.\footnote{Note, however, that $\iota \circ \psi$ is \textit{not} surjective, as the operators in $C_u^*(Y, \tilde{d})$ may move the pieces $L_n$ around.} Hence, the operator
    $$ t \colon \ell^2\left(Y\right) \rightarrow \ell^2\left(Y\right), \;\; \delta_y \mapsto
        \left\{
          \begin{array}{rl}
            \delta_{x_{k+1}^{\left(n\right)}} & \text{if} \; y = x_{k}^{\left(n\right)} \; \text{and} \; k \leq m\left(n\right) - 1, \\
            0 & \text{if} \; y = x_{m\left(n\right)}^{\left(n\right)}, \\
            \delta_y & \text{otherwise},
          \end{array}
        \right. $$
    is in the image of $\iota \circ \psi$. Now, $(Y, \tilde{d})$ is a non-extended metric space of bounded geometry, so as shown in the proof of~\cite[Theorem~2.2]{LW18} (4) $\Rightarrow$ (1), we obtain a quotient map $\xi \colon C_u^*(Y, \tilde{d}) \rightarrow B$, where the operator $t$ above maps to a proper isometry. Putting things together, $\pi := \xi \circ \iota \circ \psi \colon \unifroealg \rightarrow B$ is a homomorphism whose image contains a proper isometry. This shows that $\unifroealg$ is not strongly quasi-diagonal, and completes the proof.
  \end{proof}

  \subsection{Local \texorpdfstring{$\lrel$}{L}-finiteness, sparseness and quasi-diagonality} \label{subsec:sparse}
  The goal of this final section of the paper is to give an analogous characterization for sparse inverse semigroups.
  \begin{lemma} \label{lemma:approx-unifroealg}
    Given $\varepsilon > 0$ and $v_1, \dots, v_k \in \ell^2(S)$, there is some finite $F \Subset S$ such that if $q$ is the projection onto $\text{span}\{\delta_x \mid x \in F\}$, then $\left\|qv_i - v_i\right\| \leq \varepsilon$ for all $i = 1, \dots, k$.
  \end{lemma}
  \begin{proof}
    The proof follows from a routine $\varepsilon/2$-argument.
  \end{proof}

  \begin{theorem} \label{thm:boxspace}
    Let $S$ be a quasi-countable inverse semigroup equipped with its unique proper and right subinvariant metric $d$. The following statements are equivalent:
    \begin{enumerate}
      \item \label{thm:boxspace:lf} $S$ is locally $\lrel$-finite.
      \item \label{thm:boxspace:box} $(S, d)$ is sparse.
      \item \label{thm:boxspace:qd} $\unifroealg$ is quasi-diagonal.
      \item \label{thm:boxspace:sf} $\unifroealg$ is stably finite.
      \item \label{thm:boxspace:f} $\unifroealg$ is finite.
    \end{enumerate}
  \end{theorem}
  \begin{proof}
    The proof of this theorem is very similar to that of Theorem~\ref{thm:asymdim-zero}, both at a technical and heuristic level. We shall, hence, follow the same strategies, and show that~(\ref{thm:boxspace:lf})~$\Leftrightarrow$~(\ref{thm:boxspace:box}), and that~(\ref{thm:boxspace:lf})~$\Rightarrow$~(\ref{thm:boxspace:qd})~$\Rightarrow$~(\ref{thm:boxspace:sf})~$\Rightarrow$~(\ref{thm:boxspace:f})~$\Rightarrow$~(\ref{thm:boxspace:lf}).

    First, for the equivalence between~(\ref{thm:boxspace:lf}) and~(\ref{thm:boxspace:box}) observe that, by Lemmas~\ref{lemma:unital-technicality} and~\ref{lemma:asymdim-zero:unital}, we may assume that $S$ is a monoid. In order to prove~(\ref{thm:boxspace:lf})~$\Rightarrow$~(\ref{thm:boxspace:box}), given any $r > 0$ and $s\in S$ we will show that the $r$-component of $s$ is finite. As $d$ is proper, let $F \Subset S$ be finite and such that $y \in Fx$ for every $x, y \in S$ such that $d(x, y) \leq r$, and consider $M := \langle F \cup \{s\} \rangle \subset S$, which is an $\mathcal L$-finite inverse subsemigroup of $S$.
    
    Let $t$ be an element of $S$ in the same $r$-component. Just as in the proof of~(\ref{thm:asymdim-zero:lf})~$\Rightarrow$~(\ref{thm:asymdim-zero:zero}) in Theorem~\ref{thm:asymdim-zero}, we obtain that $t = ms$ for some $m \in M$. Note that since $s$ and $t$ must be in the same $\mathcal L$-class of $S$, we have $s^*s = t^* t = s^* m^* ms$.   Let $n=mss^* \in M$. Notice that $t = ms = mss^*s = ns$, and furthermore $n^* n=(m ss^*)^* m ss^* =s(s^* m^* m s)s^* =ss^* ss^*=ss^*$, thus $n$ is $\mathcal L$-related to $ss^*$ both in $S$ and in $M$. Denote the $\mathcal L$-class of $ss^*$ in $M$ by $L$, and note that $L$ is finite. As $n \in L$ and $t =ns$, we have obtained that $t \in Ls$ for any $t$ in the $r$-component of $s$. Thus the $r$-component of $s$ is contained in the finite set $Ls$, as required.
    
    The implication~(\ref{thm:boxspace:box})~$\Rightarrow$~(\ref{thm:boxspace:lf}) is an immediate consequence of Lemmas~\ref{lemma:finite-and-asymdim-zero}~(\ref{lemma:finite-and-asymdim-zero:l-fin}) and~\ref{lem:subsemigroup_inherits}~(\ref{lem:subsemigroup_inherits:sparse}).

    The fact that~(\ref{thm:boxspace:lf})~$\Rightarrow$~(\ref{thm:boxspace:qd}) follows from some simple approximation arguments. Indeed, observe that given $K \Subset \unifroealg, \varepsilon > 0$ and $v_1, \dots, v_k \in \ell^2(S)$, it is enough to construct a finite rank orthogonal projection $p \in \mathcal{B}(\ell^2(S))$ such that
    $$ \left|\left|pv_i - v_i\right|\right| \leq \varepsilon \;\; \text{and} \;\; \left|\left|pa - ap\right|\right| \leq \varepsilon \;\; \text{for all} \;\, a \in K \;\, \text{and} \,\; i = 1, \dots, k. $$
    By a routine $\varepsilon/2$-argument we may, without loss of generality, suppose that every $a \in K$ is of finite propagation, say bounded by $r \geq 0$. Then let $F_1 \Subset S$ be a finite set witnessing the $r$-properness of $d$. In addition, let $F_2 \Subset S$ and $q$ be as in Lemma~\ref{lemma:approx-unifroealg}, and put $T := \langle F_1 \cup F_2 \rangle$. As $S$ is locally $\lrel$-finite, observe that the $\lrel$-classes of $T$ are finite, though $T$ itself might be infinite. Consider then the subspace $V \subset \ell^2(S)$ generated by the finite set
    $$\{\delta_x: x \in T, x \lrel m\hbox{ for some }m \in F_1 \cup F_2\}.$$    
    We claim that, then, the orthogonal projection $p$ onto $V$ meets the requirements of the claim. Indeed, $p$ has finite rank, since $V$ is finite dimensional. Likewise, $p \geq q$, and hence $\left\|pv_i - v_i\right\| \leq \left\|qv_i - v_i\right\| \leq \varepsilon$ for every $i = 1, \dots, k$. Lastly, observe that $v_tp = pv_t$ for all $t \in T$. Indeed, for every $x \in S$ we have that
    $$ \left(v_t p\right) \left(\delta_x\right) = \left\{\begin{array}{rl} \delta_{tx} & \text{if} \;\; \delta_x \in V \;\, \text{and} \;\, x = t^*tx, \nonumber \\ 0 & \text{otherwise.}\end{array} \right. $$
    Likewise,
    $$ \left(p v_t\right) \left(\delta_x\right) = \left\{\begin{array}{rl} \delta_{tx} & \text{if} \;\; \delta_{tx} \in V \;\, \text{and} \;\, x = t^*tx, \nonumber \\ 0 & \text{otherwise.}\end{array} \right. $$
    In particular, observe that if $x = t^*tx$ then $x \lrel tx$, and hence $\delta_x \in V$ if, and only if, $\delta_{tx} \in V$. Thus, $pv_t = v_tp$, and hence $pa = ap$ for all $a \in K$ since, by construction, every element $a \in K$ is a finite sum of monomials of the form $fv_t$ with $t \in T$, as desired.

    The implications~(\ref{thm:boxspace:qd})~$\Rightarrow$~(\ref{thm:boxspace:sf})~$\Rightarrow$~(\ref{thm:boxspace:f}) hold for general C*-algebras (see, e.g.,~\cite[Proposition~7.1.15]{BO08}).

    Lastly, we shall show~(\ref{thm:boxspace:f})~$\Rightarrow$~(\ref{thm:boxspace:lf}). We show this by proving the contrapositive, and suppose $S$ is not locally $\lrel$-finite, that is, there is a finitely generated inverse subsemigroup $T \subset S$ and an $\lrel$-class $L \subset T$ such that $L$ is infinite. Since $T$ is finitely generated, its unique proper and right subinvariant metric $d_T$ is the path metric in the Sch\"{u}tzenberger graphs of $T$. Therefore there is an infinite $1$-path $\{x_n\}_{n \in \mathbb{N}}$, that is, $x_n \in L \subset T \subset S$ and $d_T(x_n, x_{n+1}) = 1$. Moreover, observe that this $1$-path is also a $1$-path in $(S, d)$. Hence, the operator
    $$ v \colon \ell^2\left(S\right) \rightarrow \ell^2\left(S\right), \;\; \text{where} \;\, \delta_x \rightarrow
        \left\{
          \begin{array}{rl}
            \delta_{x_{n+1}} & \text{if} \;\, x = x_n, \\
            \delta_x & \text{if} \;\, x \neq x_k \,\; \text{for any} \,\; k \in \mathbb{N},
          \end{array}
        \right. $$
    is a proper isometry, i.e., $1 = v^*v > vv^*$. Moreover, $v$ clearly has propagation at most $1$, and hence $v \in \unifroealg$, finishing the proof (see~\cite[Proposition~7.1.15]{BO08}).
  \end{proof}

  \providecommand{\bysame}{\leavevmode\hbox to3em{\hrulefill}\thinspace}
  

\begin{thebibliography}{10}
    \bibitem{ALM19} P.~Ara, F.~Lled\'{o} and D.~Mart\'{i}nez, \textit{Amenability and paradoxicality in semigroups and C*-algebras}, J. Func. Anal. \textbf{279} (2020) 108530 (43 pp).
    \bibitem{BFV20} B.~M.~Braga, I.~Farah and A.~Vignati, \textit{General uniform Roe algebra rigidity}, Ann. Inst. Fourier \textbf{72} (2022) 301--337.
    \bibitem{BO08} N.~P.~Brown and N.~Ozawa, \textit{C*-Algebras and finite-dimensional approximations}, American Mathematical Society (2008).
    \bibitem{buss-exel-2009} A.~Buss and R.~Exel, \textit{Fell bundles over inverse semigroups and twisted étale groupoids}, J. Op. Th. \textbf{67} (2012) 153--205.
    \bibitem{buss-martinez-2023} A.~Buss and D.~Mart\'{i}nez, \textit{Approximation properties of Fell bundles over inverse semigroups and non-Hausdorff groupoids}, Adv. Math. \textbf{431} (2023) 109251.
    \bibitem{DS06} A.~Dranishnikov and J.~Smith, \emph{Asymptotic dimension of discrete groups}, Fund. Math. \textbf{189}, (2006), 27--34.
    \bibitem{exel} R.~Exel, \emph{Inverse semigroups and combinatorial C*-algebras}, Bull. Braz. Math. Soc.  \textbf{39} (2008) 191--313.
    \bibitem{GSSz} R. Gray, P. V. Silva and N. Szak\'acs, \emph{Algorithmic properties of inverse monoids with hyperbolic and tree-like Sch\"utzenberger graphs}, J. Algebra (2022). 
    \bibitem{Gro81} M.~Gromov, \textit{Groups of polynomial growth and expanding maps}, Inst. Hautes \'{E}tudes Sci. Publ. Math. \textbf{53} (1981) 53--73.
    \bibitem{Gro93} M.~Gromov, \textit{Asymptotic invariants of infinite groups, in Geometric Group Theory (Sussex, 1991), vol. 2}, Cambridge University Press (1993).
    \bibitem{L98} M.~V.~Lawson, \textit{Inverse semigroups: the Theory of Partial Symmetries}, World Scientific (1998).
    \bibitem{law2} M. V. Lawson, {\em Recent developments in inverse semigroup theory}, Semigroup Forum, \textbf{100} (2020) 103--118.
    \bibitem{lawpr} M. V. Lawson, \emph{Primer on inverse semigroups} I, arXiv:2006.01628.
    
    \bibitem{LM21} F.~Lled\'{o} and D.~Mart\'{i}nez, \textit{The uniform Roe algebra of an inverse semigroup}, J. Math. Anal. Appl. \textbf{499} (2021) 124996.
    \bibitem{LL18}  K.~Li and H.-C.~Liao, \textit{Classification of uniform Roe algebras of locally finite groups}, J. Op. Th. \textbf{80} (2018) 25--46.
    \bibitem{LW18} K.~Li and R.~Willett, {\em Low dimensional properties of uniform Roe algebras}, J. London Math. Soc. \textbf{97} (2018) 98--124.
    \bibitem{MW21} X.~Ma and J.~Wu, {\em Almost elementariness and fiberwise amenability for \'{e}tale groupoids}, preprint (2020), arXiv: 2011.01182.
    \bibitem{Mos68} G.~D.~Mostow, \textit{Quasi-conformal mappings in $n$-space and the rigidity of hyperbolic space forms}, Publ. Math. Inst. Hautes \'{E}tudes Sci., \textbf{34} (1968) 53--104.
    \bibitem{Murphy1990} G.~J.~Murphy, \textit{C*-algebras and operator theory}, Academic Press, Inc. (1990).
    \bibitem{NY12} P.~W.~Nowak and G.~W.~Yu, \textit{Large Scale Geometry}, European Mathematical Society (2012).
    \bibitem{P99} A.~L.~T.~Paterson, \textit{Groupoids, Inverse Semigroups, and their Operator Algebras}, Springer Verlag (1999).
    \bibitem{Roe93} J.~Roe, \textit{Coarse cohomology and index theory on complete Riemannian manifolds}, Mem. Amer. Math. Soc., \textbf{104}(497):x+90 (1993).
    \bibitem{Roe03} J.~Roe, \textit{Index theory, coarse geometry, and topology of manifolds}, American Mathematical Society (2003).
    \bibitem{Scar17} E.~Scarparo, \textit{Characterizations of locally finite actions of groups on sets}, Glasgow Math. J. \textbf{60} (2018) 285--288.
    \bibitem{Smith} J.~Smith, \textit{On asymptotic dimension of countable abelian groups}, Top. Appl. \textbf{153} (2006) 2047--2054.
    \bibitem{steinberg-2010} B.~Steinberg, \textit{A groupoid approach to discrete inverse semigroup algebras}, Adv. Math. \textbf{223} (2010) 689--727. 
    \bibitem{S90} J.~B.~Stephen, \textit{Presentations of inverse monoids}, J. Pure and App. Alg. \textbf{6} (1990) 81--112.
    \bibitem{Stru74} R. A. Struble, \textit{Metrics in locally compact groups}, Compos. Math. \textbf{28} no. 3 (1974) 217--222.
    \bibitem{Szen} M. B. Szendrei, \textit{A note on Birget-Rhodes expansion of groups}, J. Pure Appl. Algebra, \textbf{58} (1989) 93--99.
    \bibitem{TWW17} A.~Tikuisis, S.~White and W.~Winter, \textit{Quasidiagonality of nuclear C*-algebras}, Ann. Math. \textbf{185} (2017) 229--284.
    \bibitem{Wolf68} J.~A.~Wolf, \textit{Growth of finitely generated solvable groups and curvature of Riemannian manifolds}, J. Differential Geom. \textbf{2} (1968) 421--446.
    \bibitem{Yu95} G.~Yu, \textit{Baum-Connes conjecture and coarse geometry}, K-Theory \textbf{9} (1995) 223--231.
    \bibitem{Yu98} G.~Yu, \textit{The Novikov conjecture for groups with finite asymptotic dimension}, Ann. of Math. \textbf{147} (1998) 325--355.
    \bibitem{Yu00} G.~Yu, \textit{The coarse Baum-Connes conjecture for spaces which admit a uniform embedding into Hilbert space}, Inv. Math. \textbf{139} (2000) 201--240.
  \end{thebibliography}
\end{document}